\documentclass[11pt, twoside]{amsart}

\title[Morita invariants of quasitriangular coideal subalgebras]{Morita invariants of quasitriangular coideal subalgebras}

\author{Monique M\"{u}ller}
\address{Departamento de Matem\'atica e Estat\'istica, Universidade Federal de S\~ao Jo\~ao del-Rei, Pra\c ca Frei Orlando, 170, S\~ao Jo\~ao del-Rei, MG 36307-352, Brazil \& Department of Mathematics, Indiana University, USA}
\email{monique@ufsj.edu.br}

\author{Chelsea Walton}
\address{Department of Mathematics, Rice University,
P.O. Box 1892, Houston, TX 77005-1892, USA}
\email{notlaw@rice.edu}

\usepackage{mathabx}
\usepackage{import}
\usepackage{amsmath}
\usepackage{amsfonts, stmaryrd}
\usepackage{amsthm}
\usepackage{amssymb,bbm}
\usepackage{dsfont}
\usepackage[alphabetic, initials, nobysame]{amsrefs}
\usepackage[english]{babel}
\usepackage{url}
\usepackage{fancyhdr}
\usepackage{graphicx}
\usepackage{verbatim}
\usepackage[normalem]{ulem}
\usepackage[shortlabels]{enumitem}
\newcommand{\stkout}[1]{\ifmmode\text{\sout{\ensuremath{#1}}}\else\sout{#1}\fi}
\usepackage[
colorlinks=true,
linkcolor=black, 
anchorcolor=black,
citecolor=black,
urlcolor=black, 
]{hyperref}
\usepackage[all,cmtip]{xy}
\usepackage{geometry}
\usepackage{microtype} %%Reduces bad spacing
\usepackage[dvipsnames]{xcolor}
\usepackage{tikz}

\allowdisplaybreaks

\input xy
\xyoption{all}

\definecolor{forest}{rgb}{0.0, 0.5, 0.0}

\newcommand{\bijar}[1][]{%
 \ar[#1]
 \ar@<0.7ex>@{}[#1]|-*=0[@]{\sim}} 
\usepackage{enumitem}

\DeclareMathAlphabet{\mathsf}{OT1}{cmss}{m}{n} % to math mathsf thinner

\usepackage[justification=centering]{caption}

\newcommand{\lmod}[1]{#1\text{-}\mathsf{Mod}}

\newcommand{\act}{\triangleright}
\DeclareFontFamily{U}{stixscr}{}
\DeclareFontShape{U}{stixscr}{m}{n}{<-> s*[0.8] stix-mathscr}{}

\newcommand{\Aut}{\operatorname{Aut}}

\newcommand{\ide}{\mathsf{Id}}

\newcommand\natisom{\stackrel{\hbox{$\sim$\hspace{.02in}}}{\smash{\Rightarrow}\rule{0pt}{0.4ex}}}

\newcommand{\one}{\mathds{1}}

\newcommand\equivto{\stackrel{\hbox{$\sim$\hspace{.02in}}}{\smash{\to}\rule{0pt}{0.4ex}}}
\newcommand{\eps}{\varepsilon}

\newtheoremstyle{defstyle}% name of the style to be used
  {0.5cm}                   %Space above
  {0.5cm}                   %Space below
  {\normalfont}           %Body font
  {}     %Indent amount (empty = no indent,
  {\normalfont\bfseries}  %Thm head font
  {:}                     %Punctuation after thm head
  {0.3cm}              %Space after thm head: " " = normal interword
  {\thmname{#1}\thmnumber{ #2}\thmnote{ (#3)}}
\numberwithin{equation}{section}

\newtheorem*{rep@theorem}{\rep@title}
\newcommand{\newreptheorem}[2]{%
\newenvironment{rep#1}[1]{%
 \def\rep@title{#2 \ref{##1}}%
 \begin{rep@theorem}}%
 {\end{rep@theorem}}}
\makeatother

\newtheorem{theorem}{Theorem}[section]

\newtheorem{proposition}[theorem]{Proposition}
\newreptheorem{proposition}{Proposition}
\newtheorem{corollary}[theorem]{Corollary}
\newreptheorem{corollary}{Corollary}
\newtheorem{lemma}[theorem]{Lemma}

\newtheorem{theorem*}{Theorem}
\newreptheorem{theorem}{Theorem}

\theoremstyle{definition}
\newtheorem{definition}[theorem]{Definition}

\newtheorem{example}[theorem]{Example}
\newtheorem{remark}[theorem]{Remark}

\newtheorem{question}[theorem]{Question}

\makeatletter              % This sequence of commands will
\let\c@equation\c@theorem  % incorporate equation numbering
\makeatother
\numberwithin{equation}{section}

\makeatletter
\@namedef{subjclassname@2020}{%
  \textup{2020} Mathematics Subject Classification}
\makeatother


\begin{document}

\begin{abstract}
We use representations of braid groups of Coxeter types BC and D to produce invariants of representation categories of quasitriangular coideal subalgebras. Such categories form a prevalent class of braided module categories. This is analogous to how representations of braid groups of Coxeter type A produce invariants of representation categories of quasitriangular Hopf algebras, a prevalent class of braided monoidal categories.  This work also includes concrete examples, and classification results for $K$-matrices of quasitriangular coideal subalgebras. 
\end{abstract}

\maketitle

\setcounter{tocdepth}{2}

%\begin{changemargin}{2cm}{2.5cm} 
%{\footnotesize \tableofcontents}
%\end{changemargin}

%%%%%%%%%%%%%%%%%%%%%%%%%%
%%%%%%%%%%%%%%%%%%%%%%%%%%
%%%%%%%%%%%%%%%%%%%%%%%%%%

\section{Introduction}\label{sec:intro}

The focus of our work is on the Morita equivalence for certain algebraic structures in the braided setting. For instance, recall that the category of modules over a quasitriangular Hopf algebra $(H,R)$ is a braided tensor category, where the $R$-matrix $R$ of $H$ determines precisely the braiding of $\lmod{H}$ (see, e.g., \cite[Proposition XIII.1.4]{Kassel}). Consider the following terminology.

\begin{definition}
Two quasitriangular Hopf algebras $(H,R)$ and $(H',R')$ are said to be {\it braided Morita equivalent} if their categories of modules are equivalent as braided tensor categories (i.e., there is an equivalence of categories  given by a linear, strong monoidal functor that preserves braidings).
\end{definition}

For instance, given two finite dimensional Hopf algebras $L$ and $L'$, the Drinfeld doubles $D(L)$ and $D(L')$ are braided Morita equivalent if and only if $\lmod{L}$ and $\lmod{L'}$ are {\it categorically Morita equivalent} \cite[Theorem~8.12.3]{EGNO}. See the work of Naidu \cite{Naidu} for the case when $L$ and $L'$ are finite group algebras.
Moreover, this notion of braided Morita equivalence can  be extended to quasitriangular quasi-Hopf algebras. In particular,  the braided Morita equivalence classes of twisted quantum doubles of finite groups $D^{\omega}(G)$ were determined by Naidu and Nikshych \cite[Corollary~1.5]{Naidu-Nikshych}. 
Their  representation categories are key examples of modular fusion categories; see, e.g., \cite{GruenMorrison} and the references within. 
On the other hand, in non-semisimple modular setting, Negron established a braided tensor (in fact, ribbon) equivalence between two quasi-Hopf algebras, $u_q^{\text{M}}(\mathfrak{sl}_2)$ and $u_q^{\phi}(\mathfrak{sl}_2)$,  at an even root of unity $q$  \cite[Corollary~10.4]{Negron}. Here,  $u_q^{\text{M}}(\mathfrak{sl}_2)$ is constructed from a de-equivariantization of a representation category of Lusztig's divided power algebra $U_q^{\text{L}}(\mathfrak{sl}_2)$  \cite{Lusztig}, and $u_q^{\phi}(\mathfrak{sl}_2)$ is produced by a local modules construction in \cite{CGR}. But in the works above, the standard $R$-matrix/braiding of the (quasi-)Hopf algebra/representation category is used. That said, a full classification of the quasitriangular structures on the Kac-Paljutkin Hopf algebra $H_8$, up to braided Morita equivalence, was achieved by Wakui \cite{Wakui2019} using the polynomial invariants from \cite{Wakui}; this used the  classification of $R$-matrices of $H_8$ by Suzuki \cite{Suzuki}.

\smallskip

Here, we initiate the study of a module-categorical version of braided Morita equivalence. Note that for a quasitriangular Hopf algebra $(H,R)$, the category of modules over a left $H$-comodule algebra $B$ is a left $(\lmod{H})$-module category, which is braided precisely when $B$ is equipped with a $K$-matrix~$K$ (see, e.g., \cite[Lemma 3.26]{LWY}).  In this case, we say that $(B,K)$ is a {\it quasitriangular} $H$-comodule algebra as coined by Kolb \cite{Kolb20}. Now, consider the terminology below.

\begin{definition} Fix a quasitriangular Hopf algebra $(H,R)$. We say that two quasitriangular left $H$-comodule algebras  $(B,K)$ and $(B',K')$ are  {\it braided Morita equivalent} if their categories of modules are equivalent as braided left $(\lmod{H})$-module categories (i.e., there is an equivalence of categories  given by a linear, left $(\lmod{H})$-module functor that preserves braidings).
\end{definition}

For our results on the braided Morita equivalence of quasitriangular comodule algebras, we focus on coideal subalgebras and develop equivalence class invariants analogous to those  of Shimizu for  quasitriangular Hopf algebras using braid groups \cite[Theorem~3.1]{Shimizu2010}. Indeed, braided tensor categories yield representations of braid groups of Coxeter type A, while braided module categories yield representations of braid groups of Coxeter type BC; see Propositions~\ref{prop:typeA} and~\ref{prop:typeBC}. We also develop invariants here in Coxeter type D; see Proposition~\ref{prop:typeD}. The main result of the article is presented below; here, the subscript ``\textnormal{reg}" attached to an algebra denotes the regular module.

\begin{theorem}[Theorems~\ref{thm:typeBC} and~\ref{thm:typeD}] \label{thm:main-intro}
Fix a finite-dimensional quasitriangular Hopf algebra~$H$.  
\begin{enumerate}[\upshape (a)]
\item Take quasitriangular left coideal subalgebras  $B$ and $B'$. If $\lmod{B} \simeq \lmod{B'}$ as braided left $(\lmod{H})$-module categories, then the Type BC representations 
$\rho_n^{H_\textnormal{reg}, B_\textnormal{reg}}$ and $ \rho_n^{H_\textnormal{reg}, B'_\textnormal{reg}}$  from Proposition~\ref{prop:typeBC} are equivalent, for  each $n \geq 2$.

\smallskip

\item Take triangular left coideal subalgebras  $B$ and $B'$. If $\lmod{B} \simeq \lmod{B'}$ as braided left $(\lmod{H})$-module categories, then the Type D representations 
$\rho_n^{H_\textnormal{reg}, B_\textnormal{reg}}$ and $ \rho_n^{H_\textnormal{reg}, B'_\textnormal{reg}}$  from Proposition~\ref{prop:typeD} are equivalent, for  each $n \geq 2$. \qed
\end{enumerate}
\end{theorem}

Two ingredients in the proof of Theorem~\ref{thm:main-intro} are  Andruskiewitsch--Mombelli's characterization of the Morita equivalence of comodule algebras in terms of {\it equivariant bimodules} \cite{AndruskiewitschMombelli}, and Skryabin's freeness results for such bimodules \cite{Skryabin}. In these results,  $H$-simplicity and augmentation conditions on a comodule algebra are used, which are equivalent to being a coideal subalgebra [Lemma~\ref{lem:HsimpleCSA}].  See also Propositions~\ref{prop:AM-S} and~\ref{prop:equivalence}. On the other hand, the triangularity condition above is tied to when a braided module category is {\it symmetric}; see Sections~\ref{sec:braidmm},~\ref{sec:quasitrian-ca}, and Remark~\ref{rem:allcock}.

\smallskip

Our next set of results use the invariants produced in Theorem~\ref{thm:main-intro} to classify braided Morita equivalence classes of quasitriangular coideal subalgebras over a finite group algebra and over the Sweedler algebra. Here, we classify $K$-matrices of such coideal subalgebras, which is of independent interest. Just like $R$-matrices provide solutions to the quantum Yang--Baxter equation (qYBE) (see, e.g., \cite[$\S$VIII.1]{Kassel}), $K$-matrices provide solutions to the boundary qYBE (or the quantum reflection equation); see \cite{BalagovicKolb} and references within. See also Lemma~\ref{lem:cece}, with \eqref{eq:formula-c} and \eqref{eq:formula-e}.

\begin{theorem}
[Proposition~\ref{prop:kG-K}, Theorem~\ref{thm:kG}]  \label{thm:kG-intro}
Fix a finite group $G$, and take the finite-dimensional quasitriangular Hopf algebra $\Bbbk G$ with $R$-matrix $R_u$ as in \eqref{eq:Ru}. 
Then, the statements below hold.
\begin{enumerate}[\upshape (a)]
\item Coideal subalgebras of $\Bbbk G$ are of the form $\Bbbk L$, for $L$ a subgroup of $G$.
\smallskip
\item A $K$-matrix of $\Bbbk L$ is of the form $a \otimes 1$, for $a$ in the centralizer of $L$ in $G$. 
\smallskip
\item Two quasitriangular left $\Bbbk G$-comodule algebras $(L,a \otimes 1)$ and $(L',a' \otimes 1)$ are braided Morita equivalent if and only if they are conjugate.  \qed
\end{enumerate}
\end{theorem}

\begin{theorem}[Proposition~\ref{prop:H4-K}, Theorem~\ref{thm:H4}] \label{thm:H4-intro}
Take the 4-dimensional quasitriangular Sweedler algebra $H_4$ with $R$-matrix $R_\lambda$ as in \eqref{eq:Rlambda}. 
Then, the statements below hold.
\begin{enumerate}[\upshape (a)]
\item \cite{CKS} The coideal subalgebras of $H_4$ are classified, and this includes one infinite family of proper coideal subalgebras, along with one proper coideal subalgebra outside of the family.
\smallskip
\item A $K$-matrix of a coideal subalgebra of $H_4$ must be trivial, except for $\Bbbk$ and for one of the proper coideal subalgebras in the case when $\lambda = 0$. 
\smallskip
\item There are six (resp., four) braided Morita singleton equivalence classes of quasitriangular coideal subalgebras of $H_4$ when $\lambda = 0$ (resp., when $\lambda \neq 0$), and the rest of the classes consist of pairs of quasitriangular coideal subalgebras. \qed
\end{enumerate} 
\end{theorem}

We also pose a few questions to highlight directions for further investigation. See Question~\ref{ques:hyp} (on the hypotheses of Theorem~\ref{thm:main-intro}), Question~\ref{ques:RHA} (on a notion of categorical Morita equivalence for comodule algebras), and Question~\ref{ques:othertypes} (on expanding Theorem~\ref{thm:main-intro} to other Coxeter types). More classification results akin to Theorems~\ref{thm:kG-intro} and~\ref{thm:H4-intro} are needed in the literature as well.

\medskip

\noindent{\bf Organization of the article.} We introduce preliminary material on braid groups, braided monoidal/module categories, and quasitriangular Hopf algebras/comodule algebras in Section~\ref{sec:prelim}. 
Our main results, including 
Theorem~\ref{thm:main-intro}, are established in Section~\ref{sec:results}. Then, the classification results, Theorems~\ref{thm:kG-intro} and~\ref{thm:H4-intro}, are presented in Section~\ref{sec:examples}. 

\smallskip

All categories and functors between them are assumed to be $\Bbbk$-linear, for $\Bbbk$ an algebraically closed field of characteristic zero. Linear algebraic structures are over $\Bbbk$; in this case, $\otimes := \otimes_\Bbbk$. In general, $\otimes$ denotes the bifunctor of a monoidal category, but there will be no conflict with such notation.

%%%%%%%%%%%%%%%%%%%%%%%%%%
%%%%%%%%%%%%%%%%%%%%%%%%%%
%%%%%%%%%%%%%%%%%%%%%%%%%%

\section{Preliminary material}\label{sec:prelim}

In this section, we provide background material on braid groups attached to Coxeter matrices [Section~\ref{sec:braidgrp}], on braided monoidal categories and braided module categories [Section~\ref{sec:braidmm}], and on quasitriangular Hopf algebras [Section~\ref{sec:quasitrian}], and quasitriangular comodule algebras/coideal subalgebras [Section~\ref{sec:quasitrian-ca}]. Properties of coideal subalgebras are also recalled [Section~\ref{sec:Hsimplicity}].

%%%%%%%%%%%%%%%%%%%%%%%%%%

\subsection{Braid groups attached to Coxeter matrices} 
\label{sec:braidgrp}
See  \cite[$\S\S$6.6.1,~6.6.2]{KasselTuraev} for details. 

\smallskip

A {\it Coxeter matrix} is a symmetric matrix $(a_{i,j}) \in \text{Mat}_{n}(\mathbb{N})$, where $n \in \mathbb{N}_{\geq 1}$, $a_{i,i} = 1$ for all $i$, and $a_{i,j} \in [2,\infty]$ for all $i \neq j$. Those of {\it finite type}  are divided into the following types: A$_n$, BC$_n$, D$_n$, E$_6$, E$_7$, E$_8$, F$_4$, G$_2$, H$_3$, H$_4$, and I$_2(m)$; each type is attached to a unique {\it Coxeter graph} as visualized in \cite[Table~6.1]{KasselTuraev}. Each Coxeter matrix $(a_{i,j}) \in \text{Mat}_{n}(\mathbb{N})$ of type X$_n$ is also associated to a unique group generated by $\theta_1, \dots, \theta_n$, subject to the relations:
\[
\underbrace{\theta_i \theta_j \theta_i \dots}_{a_{ij}\, \text{factors}} \; = \; \underbrace{\theta_j \theta_i \theta_i \dots}_{a_{ij}\, \text{factors}}.
\]
We refer to this group as the {\it braid group of Coxeter type \textnormal{X}} (or {\it type $\textnormal{X}_n$}).

\smallskip

We present the matrices, graphs, and braid groups of Coxeter types A, BC, and D next.

\subsubsection{Coxeter type \textnormal{A}}
The Coxeter matrix and graph of type A$_n$ are given below.
\smallskip
\begin{center}
\scalebox{0.8}{
$
\begingroup
\setlength\arraycolsep{-0.5pt}
\begin{pmatrix}
\; 1\;  & 3 & 2 & \dots &\; \dots & 2\\[-.6pc]
3 & 1 & 3 & \ddots &  & \vdots\\[-.6pc]
2 & 3 & 1 & \ddots & \ddots & \vdots\\[-.6pc]
\vdots & \ddots & \ddots  & \ddots & \ddots  & 2\\[-.6pc]
\vdots &  & \ddots & \ddots  & \ddots & 3\\[-.2pc]
2 & \dots\; & \dots & 2 & 3 & \; 1\; 
\end{pmatrix}
\endgroup
$
}
\hspace{0.5in}
\tikzset{every picture/.style={line width=0.75pt}} %set default line width to 0.75pt        
\begin{tikzpicture}[x=0.75pt,y=0.75pt,yscale=-1,xscale=1]
%uncomment if require: \path (0,61); %set diagram left start at 0, and has height of 61
%Shape: Circle [id:dp14500678826010605] 
\draw   (11.25,23.76) .. controls (11.25,21.96) and (12.7,20.51) .. (14.49,20.51) .. controls (16.29,20.51) and (17.74,21.96) .. (17.74,23.76) .. controls (17.74,25.55) and (16.29,27) .. (14.49,27) .. controls (12.7,27) and (11.25,25.55) .. (11.25,23.76) -- cycle ;
%Shape: Circle [id:dp9600656995059391] 
\draw   (161.25,23.51) .. controls (161.25,21.71) and (162.7,20.26) .. (164.49,20.26) .. controls (166.29,20.26) and (167.74,21.71) .. (167.74,23.51) .. controls (167.74,25.3) and (166.29,26.75) .. (164.49,26.75) .. controls (162.7,26.75) and (161.25,25.3) .. (161.25,23.51) -- cycle ;
%Straight Lines [id:da7321923641541868] 
\draw    (17.74,23.76) -- (161.25,23.76) ;
%Shape: Circle [id:dp46213070131669176] 
\draw  [fill={rgb, 255:red, 255; green, 255; blue, 255 }  ,fill opacity=1 ] (41,23.76) .. controls (41,21.96) and (42.45,20.51) .. (44.24,20.51) .. controls (46.04,20.51) and (47.49,21.96) .. (47.49,23.76) .. controls (47.49,25.55) and (46.04,27) .. (44.24,27) .. controls (42.45,27) and (41,25.55) .. (41,23.76) -- cycle ;
%Shape: Circle [id:dp009197804595122872] 
\draw  [fill={rgb, 255:red, 255; green, 255; blue, 255 }  ,fill opacity=1 ] (71.25,23.51) .. controls (71.25,21.71) and (72.7,20.26) .. (74.49,20.26) .. controls (76.29,20.26) and (77.74,21.71) .. (77.74,23.51) .. controls (77.74,25.3) and (76.29,26.75) .. (74.49,26.75) .. controls (72.7,26.75) and (71.25,25.3) .. (71.25,23.51) -- cycle ;
%Shape: Circle [id:dp6975814878491551] 
\draw  [fill={rgb, 255:red, 255; green, 255; blue, 255 }  ,fill opacity=1 ] (101.25,23.76) .. controls (101.25,21.96) and (102.7,20.51) .. (104.49,20.51) .. controls (106.29,20.51) and (107.74,21.96) .. (107.74,23.76) .. controls (107.74,25.55) and (106.29,27) .. (104.49,27) .. controls (102.7,27) and (101.25,25.55) .. (101.25,23.76) -- cycle ;
%Shape: Rectangle [id:dp588170307315419] 
\draw  [color={rgb, 255:red, 255; green, 255; blue, 255 }  ,draw opacity=1 ][fill={rgb, 255:red, 255; green, 255; blue, 255 }  ,fill opacity=1 ] (117.94,19.28) -- (150.94,19.28) -- (150.94,29.78) -- (117.94,29.78) -- cycle ;
% Text Node
\draw (120.4, 21.4) node [anchor=north west][inner sep=0.75pt]   [align=left] {. . .};
% Text Node
\draw (10.75,31.96) node [anchor=north west][inner sep=0.75pt]  [font=\scriptsize] [align=left] {1 \hspace{0.165in} 2 \hspace{0.175in} 3 \hspace{0.175in} 4 \hspace{0.48in} $n$};
\end{tikzpicture}
\end{center}

\smallskip

\noindent The corresponding braid group here is:
\[
\textnormal{Br}_n^{\textnormal{A}} :=
\Big\langle \sigma_1, \dots, \sigma_n ~~\Big| \hspace{-.1in}
\begin{array}{cl}
\hspace{0.15in} \sigma_i \sigma_{i+1} \sigma_i = \sigma_{i+1} \sigma_i \sigma_{i+1} &\text{for $i=1, \dots, n-1$}\,;\\[0.03pc]
\sigma_i \sigma_j =  \sigma_j \sigma_i &\text{for $|i-j| \geq 2$}  
\end{array}
\Big\rangle.
\]

\subsubsection{Coxeter type \textnormal{BC}}
The Coxeter matrix and graph of type BC$_n$ are given below.
\smallskip
\begin{center}
\scalebox{0.8}{
$
\begingroup
\setlength\arraycolsep{-0.5pt}
\begin{pmatrix}
\; 1\;  & 3 & 2 & \dots & \;\dots & \;\dots & 2\\[-.6pc]
3 & 1 & 3 & \ddots &  & & \vdots\\[-.6pc]
2 & 3 & 1 & \ddots & \ddots & & \vdots\\[-.6pc]
\vdots & \ddots & \ddots  & \ddots & \ddots & \ddots  & \vdots\\[-.6pc]
\vdots &  & \ddots & \ddots  & 1 &3 & 2\\[-.6pc]
\vdots & \dots\; & \dots  & \ddots & 3 & 1 & \; 4\;
\\[-.2pc]
2 & \dots \; & \dots\; & \dots & 2 & 4 & \; 1\; 
\end{pmatrix}
\endgroup
$
}
\hspace{0.5in}
\tikzset{every picture/.style={line width=0.75pt}} %set default line width to 0.75pt        
\begin{tikzpicture}[x=0.75pt,y=0.75pt,yscale=-1,xscale=1]
%uncomment if require: \path (0,61); %set diagram left start at 0, and has height of 61
%Shape: Circle [id:dp14500678826010605] 
\draw   (11.25,23.76) .. controls (11.25,21.96) and (12.7,20.51) .. (14.49,20.51) .. controls (16.29,20.51) and (17.74,21.96) .. (17.74,23.76) .. controls (17.74,25.55) and (16.29,27) .. (14.49,27) .. controls (12.7,27) and (11.25,25.55) .. (11.25,23.76) -- cycle ;
%Shape: Circle [id:dp9600656995059391] 
\draw   (161.25,23.51) .. controls (161.25,21.71) and (162.7,20.26) .. (164.49,20.26) .. controls (166.29,20.26) and (167.74,21.71) .. (167.74,23.51) .. controls (167.74,25.3) and (166.29,26.75) .. (164.49,26.75) .. controls (162.7,26.75) and (161.25,25.3) .. (161.25,23.51) -- cycle ;
%Straight Lines [id:da7321923641541868] 
\draw    (17.74,23.76) -- (161.25,23.76) ;
%Shape: Circle [id:dp46213070131669176] 
\draw  [fill={rgb, 255:red, 255; green, 255; blue, 255 }  ,fill opacity=1 ] (41,23.76) .. controls (41,21.96) and (42.45,20.51) .. (44.24,20.51) .. controls (46.04,20.51) and (47.49,21.96) .. (47.49,23.76) .. controls (47.49,25.55) and (46.04,27) .. (44.24,27) .. controls (42.45,27) and (41,25.55) .. (41,23.76) -- cycle ;
%Shape: Circle [id:dp009197804595122872] 
\draw  [fill={rgb, 255:red, 255; green, 255; blue, 255 }  ,fill opacity=1 ] (71.25,23.51) .. controls (71.25,21.71) and (72.7,20.26) .. (74.49,20.26) .. controls (76.29,20.26) and (77.74,21.71) .. (77.74,23.51) .. controls (77.74,25.3) and (76.29,26.75) .. (74.49,26.75) .. controls (72.7,26.75) and (71.25,25.3) .. (71.25,23.51) -- cycle ;
%Shape: Circle [id:dp6975814878491551] 
\draw  [fill={rgb, 255:red, 255; green, 255; blue, 255 }  ,fill opacity=1 ] (101.25,23.76) .. controls (101.25,21.96) and (102.7,20.51) .. (104.49,20.51) .. controls (106.29,20.51) and (107.74,21.96) .. (107.74,23.76) .. controls (107.74,25.55) and (106.29,27) .. (104.49,27) .. controls (102.7,27) and (101.25,25.55) .. (101.25,23.76) -- cycle ;
%Shape: Rectangle [id:dp588170307315419] 
\draw  [color={rgb, 255:red, 255; green, 255; blue, 255 }  ,draw opacity=1 ][fill={rgb, 255:red, 255; green, 255; blue, 255 }  ,fill opacity=1 ] (117.94,19.28) -- (150.94,19.28) -- (150.94,29.78) -- (117.94,29.78) -- cycle ;
%Shape: Circle [id:dp697556066709454] 
\draw   (190.92,23.51) .. controls (190.92,21.71) and (192.37,20.26) .. (194.16,20.26) .. controls (195.95,20.26) and (197.4,21.71) .. (197.4,23.51) .. controls (197.4,25.3) and (195.95,26.75) .. (194.16,26.75) .. controls (192.37,26.75) and (190.92,25.3) .. (190.92,23.51) -- cycle ;
%Straight Lines [id:da41849833130200387] 
\draw    (167.12,23.76) -- (192.12,23.76) ;
%Straight Lines [id:da5198996235329145] 
%\draw    (166.78,25.68) -- (191.78,25.68) ;
% Text Node
\draw (120.4, 21.4) node [anchor=north west][inner sep=0.75pt]  [align=left] {. . .};
% Text Node
\draw (176, 9) node [anchor=north west][inner sep=0.75pt]  [font=\scriptsize]  [align=left] {4};
% Text Node
\draw (10.75,31.96) node [anchor=north west][inner sep=0.75pt]  [font=\scriptsize] [align=left] {1 \hspace{0.165in} 2 \hspace{0.175in} 3 \hspace{0.175in} 4 \hspace{0.39in} $n \hspace{-0.02in} - \hspace{-0.02in} 1$ \hspace{0.09in} $n$};
\end{tikzpicture}
\end{center}

\smallskip

\noindent The corresponding braid group here is:
\[
\textnormal{Br}_n^{\textnormal{BC}} :=
\left\langle \sigma_1, \dots, \sigma_{n-1}, t ~~\Bigg| \hspace{-.1in}
\begin{array}{cl}
\hspace{0.3in} \sigma_i \sigma_{i+1} \sigma_i = \sigma_{i+1} \sigma_i \sigma_{i+1} &\text{for $i=1, \dots, n-2$}\,;\\[-0.15pc]
\hspace{0.15in} \sigma_i \sigma_j =  \sigma_j \sigma_i &\text{for $|i-j| \geq 2$}\,;\\[-0.15pc]
\hspace{0.15in} \sigma_i \, t = t  \, \sigma_i  &\text{for $i=1, \dots, n-2$}\,;\\[-0.15pc]
\hspace{0.15in} \sigma_{n-1} \, t \, \sigma_{n-1} \, t = t  \, \sigma_{n-1} \, t \, \sigma_{n-1}
\end{array}
\right\rangle.
\]

\subsubsection{Coxeter type \textnormal{D}}
The Coxeter matrix and graph of type D$_n$ are given below.
\smallskip
\begin{center}
\scalebox{0.8}{
$
\begingroup
\setlength\arraycolsep{-0.5pt}
\begin{pmatrix}
\; 1\;  & 3 & 2 & 2 & \; \dots & \;\dots & \;\dots & \;\dots & 2\\[-.6pc]
3 & 1 & 3 & 2& \ddots & &  & & \vdots\\[-.6pc]
2 & 3 & 1 & 3& \ddots & \ddots & & & \vdots\\[-.6pc]
2 & 2 & 3  & 1& \ddots & \ddots & \ddots &  & \vdots\\[-.6pc]
\vdots & \ddots & \ddots\; & \ddots & \ddots & \ddots & \ddots & \ddots& \vdots\\[-.6pc]
\vdots &  &\ddots & \ddots & \ddots & 1  & 3 &2 & 2\\[-.6pc]
\vdots &  & & \ddots & \ddots & 3  & 1 &3 & 3\\[-.6pc]
\vdots &  &  && \ddots & 2 & 3 & 1 & \; 2\;\\[-.2pc]
2 & \dots \; & \dots\; & \dots\; & \dots & 2 & 3 & 2 & \; 1\; 
\end{pmatrix}
\endgroup
$
}
\hspace{0.5in}
\tikzset{every picture/.style={line width=0.75pt}} %set default line width to 0.75pt        
\begin{tikzpicture}[x=0.75pt,y=0.75pt,yscale=-1,xscale=1]
%uncomment if require: \path (0,61); %set diagram left start at 0, and has height of 61
%Shape: Circle [id:dp14500678826010605] 
\draw   (11.25,23.76) .. controls (11.25,21.96) and (12.7,20.51) .. (14.49,20.51) .. controls (16.29,20.51) and (17.74,21.96) .. (17.74,23.76) .. controls (17.74,25.55) and (16.29,27) .. (14.49,27) .. controls (12.7,27) and (11.25,25.55) .. (11.25,23.76) -- cycle ;
%Shape: Circle [id:dp9600656995059391] 
\draw   (161.25,23.51) .. controls (161.25,21.71) and (162.7,20.26) .. (164.49,20.26) .. controls (166.29,20.26) and (167.74,21.71) .. (167.74,23.51) .. controls (167.74,25.3) and (166.29,26.75) .. (164.49,26.75) .. controls (162.7,26.75) and (161.25,25.3) .. (161.25,23.51) -- cycle ;
%Straight Lines [id:da7321923641541868] 
\draw    (17.74,23.76) -- (161.25,23.51) ;
%Shape: Circle [id:dp46213070131669176] 
\draw  [fill={rgb, 255:red, 255; green, 255; blue, 255 }  ,fill opacity=1 ] (41,23.76) .. controls (41,21.96) and (42.45,20.51) .. (44.24,20.51) .. controls (46.04,20.51) and (47.49,21.96) .. (47.49,23.76) .. controls (47.49,25.55) and (46.04,27) .. (44.24,27) .. controls (42.45,27) and (41,25.55) .. (41,23.76) -- cycle ;
%Shape: Circle [id:dp009197804595122872] 
\draw  [fill={rgb, 255:red, 255; green, 255; blue, 255 }  ,fill opacity=1 ] (71.25,23.51) .. controls (71.25,21.71) and (72.7,20.26) .. (74.49,20.26) .. controls (76.29,20.26) and (77.74,21.71) .. (77.74,23.51) .. controls (77.74,25.3) and (76.29,26.75) .. (74.49,26.75) .. controls (72.7,26.75) and (71.25,25.3) .. (71.25,23.51) -- cycle ;
%Shape: Circle [id:dp6975814878491551] 
\draw  [fill={rgb, 255:red, 255; green, 255; blue, 255 }  ,fill opacity=1 ] (101.25,23.76) .. controls (101.25,21.96) and (102.7,20.51) .. (104.49,20.51) .. controls (106.29,20.51) and (107.74,21.96) .. (107.74,23.76) .. controls (107.74,25.55) and (106.29,27) .. (104.49,27) .. controls (102.7,27) and (101.25,25.55) .. (101.25,23.76) -- cycle ;
%Shape: Rectangle [id:dp588170307315419] 
\draw  [color={rgb, 255:red, 255; green, 255; blue, 255 }  ,draw opacity=1 ][fill={rgb, 255:red, 255; green, 255; blue, 255 }  ,fill opacity=1 ] (117.94,19.28) -- (150.94,19.28) -- (150.94,29.78) -- (117.94,29.78) -- cycle ;
%Shape: Circle [id:dp697556066709454] 
\draw   (190.58,41.51) .. controls (190.58,39.71) and (192.04,38.26) .. (193.83,38.26) .. controls (195.62,38.26) and (197.07,39.71) .. (197.07,41.51) .. controls (197.07,43.3) and (195.62,44.75) .. (193.83,44.75) .. controls (192.04,44.75) and (190.58,43.3) .. (190.58,41.51) -- cycle ;
%Straight Lines [id:da41849833130200387] 
\draw    (167.74,23.51) -- (190.5,6.75) ;
%Straight Lines [id:da5198996235329145] 
\draw    (167.74,23.51) -- (190.58,41.51) ;
%Shape: Circle [id:dp3914619386586322] 
\draw   (190.5,6.75) .. controls (190.5,4.96) and (191.95,3.5) .. (193.74,3.5) .. controls (195.53,3.5) and (196.99,4.96) .. (196.99,6.75) .. controls (196.99,8.54) and (195.53,9.99) .. (193.74,9.99) .. controls (191.95,9.99) and (190.5,8.54) .. (190.5,6.75) -- cycle ;

% Text Node
\draw (120.4, 21.4) node [anchor=north west][inner sep=0.75pt]   [align=left] {. . .};
% Text Node
\draw (10.75,31.96) node [anchor=north west][inner sep=0.75pt]  [font=\scriptsize] [align=left] {1 \hspace{0.165in} 2 \hspace{0.175in} 3 \hspace{0.175in} 4 \hspace{0.39in} $n \hspace{-0.02in} - \hspace{-0.02in} 2$};
% Text Node
\draw (182,12.32) node [anchor=north west][inner sep=0.75pt]  [font=\scriptsize] [align=left] {$n \hspace{-0.02in} - \hspace{-0.02in}1$};
% Text Node
\draw (189,50) node [anchor=north west][inner sep=0.75pt]  [font=\scriptsize] [align=left] {$n$};
\end{tikzpicture}
\end{center}

\smallskip

\noindent The corresponding braid group here is:
\[
\textnormal{Br}_n^{\textnormal{D}} :=
\Bigg\langle \sigma_1, \dots, \sigma_{n-1}, t ~~\Bigg| \hspace{-.1in}
\begin{array}{cl}
\hspace{0.5in} \sigma_i \sigma_{i+1} \sigma_i = \sigma_{i+1} \sigma_i \sigma_{i+1} &\text{for $i=1, \dots, n-2$}\, ;\\[-0.15pc]
\hspace{0.37in} \sigma_i \sigma_j =  \sigma_j \sigma_i &\text{for $|i-j| \geq 2$}\, ;\\[0pc]
 \hspace{0.37in} \sigma_i \, t = t  \, \sigma_i  &\text{for $i=1, \dots, n-3,\, n-1$}\,;\\[-0.15pc]
\hspace{0.1in} \sigma_{n-2} \, t \, \sigma_{n-2}  = t  \, \sigma_{n-2} \, t  
\end{array}
\Bigg\rangle.
\]

%%%%%%%%%%%%%%%%%%%%%%%%%%

\subsection{Braided monoidal categories and braided module categories} \label{sec:braidmm} 
We refer the reader to \cite[Chapter~3]{Walton2024}, \cite[\S 8.1]{EGNO}, and \cite[\S 3.5]{LWY} for further details.

\smallskip

A {\it monoidal category} consists of a category~$\mathcal{C}$ equipped with a bifunctor $\otimes:  \mathcal{C} \times \mathcal{C} \to \mathcal{C}$ and an object $\one \in \mathcal{C}$, along with natural isomorphisms that serve as associativity and unitality constraints  satisfying pentagon and triangle axioms. Here, we will assume by way of the Mac Lane's strictness theorem that monoidal categories are {\it strict}---that is, the components of the associativity and unitality constraints are identity morphisms.

\smallskip

A {\it (strong) monoidal functor}
between monoidal categories $(\mathcal{C}, \otimes, \one)$ and $(\mathcal{C}', \otimes', \one')$ is a functor $F: \mathcal{C} \to \mathcal{C}'$ equipped with a natural isomorphism $F^{(2)}:=\{F^{(2)}_{X,Y}:   F(X) \otimes' F(Y) \equivto F(X \otimes Y)\}_{X,Y \in \mathcal{C}}$, and an isomorphism $F^{(0)} : \one'  \equivto F(\one)$  in $\mathcal{C}'$, satisfying associativity and unitality constraints. 

\smallskip

For a monoidal category $(\mathcal{C}, \otimes, \one)$, a {\it left $\mathcal{C}$-module category}  is a category $\mathcal{M}$ equipped with
a bifunctor $\act\colon \mathcal{C} \times \mathcal{M} \to \mathcal{M}$ (an {\it action}),  with natural transformations that serve as associativity and unitality constraints  satisfying pentagon and triangle axioms. Here, we will also assume by way of a version of Mac Lane's strictness theorem that module categories are {\it strict}, that is, the components of the associativity and unitality constraints are identity morphisms.

\smallskip

A {\it $\mathcal{C}$-module functor} from $(\mathcal{M}, \act)$ to $(\mathcal{M}', \act')$ is a functor $F: \mathcal{M} \to \mathcal{M}'$ equipped with a natural isomorphism \! $s \!:= \! {\{s_{X,M}\colon \!F(X \! \act \! M) \!\equivto \! X\! \act'\! F(M)\}}_{X \in \mathcal{C}, M \in \mathcal{M}}$ \! satisfying pentagon and triangle~axioms.

\smallskip

A {\it braided monoidal category} is a monoidal category $(\mathcal{C}, \otimes, \one)$ equipped with a family of natural isomorphisms $c:=\{c_{X, Y}: X\otimes Y\equivto Y\otimes X\}_{X, Y\in \mathcal{C}}$, called a {\it braiding}, satisfying the hexagon axioms: 
\begin{eqnarray}
c_{X, Y\otimes Z}& = (\ide_Y\otimes c_{X, Z})(c_{X, Y}\otimes\ide_Z), \label{eq:braid1} \\
c_{X\otimes Y, Z}& = (c_{X, Z}\otimes \ide_Y)(\ide_X\otimes c_{Y, Z}),\label{eq:braid2}
\end{eqnarray} for all $X, Y, Z\in\mathcal{C}$. Further, a braided monoidal category $(\mathcal{C}, \otimes, \one, c)$ is said to be {\it symmetric} if $c_{X,Y} = c^{-1}_{Y,X}$, for all $X, Y \in \mathcal{C}$. 

\smallskip

A {\it braided monoidal functor} between braided monoidal categories $(\mathcal{C}, \otimes, \one, c)$ and $(\mathcal{C}',  \otimes', \one', c')$ is a monoidal functor $(F, F^{(2)},F^{(0)}): \mathcal{C}\to \mathcal{C}'$ such that 
\[
F^{(2)}_{Y, X} \, c'_{F(X), F(Y)} \; = \; F(c_{X, Y}) \, F^{(2)}_{X, Y},
\]
 for all $X, Y\in \mathcal{C}$. 
A {\it braided monoidal equivalence} of braided monoidal categories is a braided
monoidal functor which is also an equivalence of categories.

\smallskip

Let $\mathcal{C}$ be a braided monoidal category with braiding $c$. A left $\mathcal{C}$-module category $(\mathcal{M}, \act)$ is called {\it braided} if it is equipped with a natural isomorphism $e:=\{e_{X, M}: X\act M\equivto X\act M\}_{X\in \mathcal{C}, M\in \mathcal{M}}$, such that the following identities hold for all $X, Y\in \mathcal{C}$, $M \in \mathcal{M}$:
\begin{eqnarray} 
  e_{X\otimes Y, M} & = & (\ide_X \otimes e_{Y, M})(c_{Y, X}\act \ide_M)(\ide_Y \otimes e_{X, M})(c^{-1}_{Y, X}\act \ide_M), \label{eq:brmod1}\\ 
   e_{X, Y\act M}& = &  (c_{Y, X}\act \ide_M) (\ide_Y \otimes e_{X, M})(c_{X, Y}\act \ide_M). \label{eq:brmod2}
\end{eqnarray} 
We call a braided left $\mathcal{C}$-module category $\mathcal{M}$ {\it symmetric} if $e_{X, M} = e^{-1}_{X, M}$, for all $X\in\mathcal{C}$, $M\in\mathcal{M}$.

\smallskip

A {\it braided $\mathcal{C}$-module functor} between braided left $\mathcal{C}$-module categories $(\mathcal{M}, \act, e)$ and $(\mathcal{M}', \act', e')$ is a left $\mathcal{C}$-module functor $(F, s): (\mathcal{M}, \act)\to (\mathcal{M}', \act')$ such that 
\begin{equation} \label{eq:brmodfun}
e'_{X, F(M)} \, s_{X, M} \; = \; s_{X, M} \, F(e_{X, M}),
\end{equation}
 for all $X\in \mathcal{C}$, $M\in \mathcal{M}$. An {\it equivalence of braided left $\mathcal{C}$-module categories} is a braided module functor which is also an equivalence of the underlying categories.

%%%%%%%%%%%%%%%%%%%%%%%%%%

\subsection{Quasitriangular Hopf algebras} 
\label{sec:quasitrian} See \cite[$\S$8.3]{EGNO} for details about the material below.  

\smallskip

Recall that a Hopf algebra $(H, \, \Delta \colon H \to H \otimes H,  \, \varepsilon \colon H \to \Bbbk,  \, S \colon H \to H)$ has a monoidal category of left $H$-modules, $\lmod{H}$, with $\otimes$ and $\one$ determined by $\Delta$ and $\varepsilon$, respectively.

\smallskip

A Hopf algebra $H$ is {\it quasitriangular} if there exists an invertible element $R\in H\otimes H$, called an {\it $R$-matrix}, such that, for all $h \in H$,
\begin{equation}
  \text{(i)} \,  (\Delta\otimes\ide_H)(R)= R_{13}R_{23}, \quad \; \;  \text{(ii)} \,  (\ide_H\otimes \Delta)(R) = R_{13}R_{12},   \quad \; \; \text{(iii)} \, R \Delta(h) =\Delta^{\operatorname{op}}(h)R.
\end{equation} 
Further, a quasitriangular Hopf algebra $(H,R)$ is said to be {\it triangular} if $R_{21} = R^{-1}$. 

\smallskip

One  quasitriangular Hopf algebra is the {\it Drinfeld double} $D(L)$ of a finite-dimensional Hopf algebra~$L$. As a coalgebra, $D(L) = (L^*)^{\text{op}} \otimes L$, and
 $\ell \xi = \langle \xi_{(1)}, S(\ell_{(1)}) \rangle \langle \xi_{(3)}, \ell_{(2)} \rangle \xi_{(2)} \ell_{(2)}$ for $\ell \in L$ and $\xi \in L^*$. Here, we employ the Sweedler notation:
\[
\Delta(\ell):= \ell_{(1)} \otimes \ell_{(2)} \in H \otimes H.
\] 
Also, $D(L)$ has an $R$-matrix $\sum_d 1_{L^*} \otimes \ell_d \otimes \xi_d \otimes 1_L$, for $\{\ell_d, \xi_d\}$ a dual basis of $L$.
In fact, $\lmod{D(L)}$ is equivalent to the {\it Drinfeld center} $\mathcal{Z}(\lmod{L})$ as braided monoidal categories.

\smallskip

Given a Hopf algebra $H$, there is a bijection between braidings of the monoidal category of left $H$-modules and $R$-matrices of $H$. More precisely, for all $X, Y \in \lmod{H}$, 
\begin{equation} \label{eq:formula-c}
c_{X,Y}: X\otimes Y\to Y\otimes X, \quad x\otimes y \mapsto \textstyle \sum_{i} (R^i \cdot y) \otimes (R_i\cdot x),
\end{equation}
gives a braiding of $\lmod{H}$ if and only if $R=\sum_i R_i\otimes R^i \in H\otimes H$ is an $R$-matrix of $H$. This braiding of $\lmod{H}$ is symmetric if and only if $(H,R)$ is triangular.

\smallskip

We say that quasitriangular Hopf algebras $(H,R)$ and $(H',R')$ are {\it braided Morita equivalent} if $\lmod{H}$ and $\lmod{H'}$ are equivalent as braided monoidal categories.

%%%%%%%%%%%%%%%%%%%%%%%%%%

\subsection{Quasitriangular comodule algebras and quasitriangular coideal subalgebras} 
\label{sec:quasitrian-ca} For details of the material below, see \cite[\S\S2.2, 2.3]{Kolb20} and \cite[$\S$3.7]{LWY}. 

\smallskip

Recall that for a Hopf algebra $H$, the category of modules over a left $H$-comodule algebra $(B, \, \delta \colon B \to H \otimes B)$ is a left $(\lmod{H})$-module category, $\lmod{B}$, with $\act$  determined by $\delta$.

\smallskip

Let $H$ be a quasitriangular Hopf algebra with  $R$-matrix $R=\sum_i R_i\otimes R^i \in H\otimes H$. A left $H$-comodule algebra $(B,\delta \colon B \to H \otimes B)$ is {\it quasitriangular} if there exists an invertible element $K \in H \otimes B$, called a {\it $K$-matrix}, such that, for all $b \in B$,
\begin{equation}\label{eq:K-matrix}
  \text{(i)} \,  (\Delta \otimes \ide_B)(K)= K_{23}R_{21}K_{13}R_{21}^{-1}, \quad \;   \text{(ii)} \,  (\ide_H\otimes \delta)(K) = R_{21}K_{13}R_{12},   \quad \;  \text{(iii)} \, K \delta(b) =\delta(b)K.
\end{equation} 
We say that a quasitriangular left $H$-comodule algebra $(B,K)$ is  {\it triangular} if $K = K^{-1}$. 

\smallskip

In fact, there is a simpler criteria for a coideal subalgebra to be quasitriangular.
Let $(H,R)$ be a quasitriangular Hopf algebra.  A left coideal subalgebra $(B,\Delta|_B)$ is said to be {\it quasitriangular} if there exists an invertible element $\widehat{K} \in H$, such that
\begin{equation}\label{eq:K-matrix-coideal-sub}
  (\widehat{\text{i}}) \,  \Delta(\widehat{K})= \widehat{K}_2 R_{21}\widehat{K}_1 R_{21}^{-1}, \quad \quad   (\widehat{\text{ii}}) \,   R_{21}\widehat{K}_1R_{12}\in H\otimes B,   \quad \quad  (\widehat{\text{iii}}) \, \widehat{K} b =b\widehat{K},\, \forall b\in B.
\end{equation} 

\medskip

The next lemma can be proved in a similar way as \cite[Lemma 2.9]{Kolb20}. 

\begin{lemma}\label{lemma:criteria-K-matrix} 
    Let $(H,R)$ be a quasitriangular Hopf algebra.
    \begin{enumerate}[\upshape (a)]
        \item If $B$ is a left coideal subalgebra of $H$ that is quasitriangular as an $H$-comodule algebra with $K$-matrix $K\in H\otimes B$, then $B$ is quasitriangular as a left coideal subalgebra with $\widehat{K}=(\ide\otimes\varepsilon)(K)\in H$. In this case, $K=R_{21}\widehat{K}_1R_{12}$.
        \smallskip
        \item Conversely, if $B$ is quasitriangular as a coideal subalgebra with invertible element $\widehat{K}$, then $B$ is quasitriangular as $H$-comodule algebra with $K$-matrix $K=R_{21}\widehat{K}_1R_{12}$. \qed
    \end{enumerate}
\end{lemma}

\smallskip

One quasitriangular left $H$-comodule algebra is the {\it reflective algebra} $R_H(A)$ of a finite-dimensional \linebreak quasitriangular Hopf algebra~$H$ and a left $H$-comodule algebra $A$.
Here, $R_H(A) = A \otimes H^*$ as a vector space, and $\xi a = a_{[0]}(\xi \leftharpoonup a_{[-1]})$ for $a \in A$ and $\xi \in H^*$, where $\langle \xi \leftharpoonup \ell, h\rangle := \langle \xi, \ell_{(2)} h S^{-1}(\ell_{(1)}) \rangle$ for $h,\ell \in H$. Here, we also employ the Sweedler notation:
\[
\delta(a):= a_{[-1]} \otimes a_{[0]} \in H \otimes A.\] 
Also, $R_H(A)$ has a $K$-matrix $\sum_d  h_d \otimes 1_A \otimes \xi_d$, for $\{h_d, \xi_d\}$ a dual basis of $H$. In fact, $\lmod{R_H(A)}$ is equivalent to the {\it reflective center} $\mathcal{E}_{\lmod{H}}(\lmod{A})$ as braided left $(\lmod{H})$-module categories.
 
 \smallskip
 
Given a quasitriangular Hopf algebra $(H, R)$ and a left $H$-comodule algebra $B$, there is a bijection between braidings of the left $(\lmod{H})$-module category of left $B$-modules and $K$-matrices of $B$. 
More precisely, for all $(X, \cdot) \in \lmod{H}$ and $(M, \ast) \in \lmod{B}$,  
\begin{align}\label{eq:formula-e}
    e_{X, M}: X \act M \to X \act M, \quad x \otimes m \mapsto \textstyle \sum_{i} (K_i\cdot x) \otimes (K^i \ast m)
\end{align}
 is a braiding of $\lmod{B}$ if and only if $K=\sum_i K_i\otimes K^i \in H\otimes B$ is a $K$-matrix of $B$.  This braiding of $\lmod{B}$ is symmetric if and only if  $(B,K)$ is triangular.
 
 \smallskip
 
 We say that quasitriangular left $H$-comodule algebras  $(B,K)$ and $(B',K')$ are {\it braided Morita equivalent} if $\lmod{B}$ and $\lmod{B'}$ are equivalent as braided left $(\lmod{H})$-module categories.

%%%%%%%%%%%%%%%%%%%%%%%%%%

\subsection{Properties of coideal subalgebras} 
\label{sec:Hsimplicity}
Fix a Hopf algebra $(H, \Delta, \eps, S)$ and a left $H$-comodule algebra $(A, \delta)$. 
We say that $A$ is {\it augmented} if there exists an algebra map $\eps_A: A \to \Bbbk$ ({\it augmentation map}). For example, $H$ is augmented with the augmentation map being its counit map $\varepsilon$.
On the other hand, an {\it $H$-costable ideal} of $A$ is an ideal $I$ of $A$ such that $\delta(I)\subseteq H\otimes I$. The comodule algebra $A$ is called {\it $H$-simple} if has no non-trivial $H$-costable ideal. 
In fact, we have the result below.

\begin{lemma} \label{lem:HsimpleCSA}
Take a finite-dimensional Hopf algebra $H$. Then, a left $H$-comodule algebra $(A,\delta)$ is $H$-simple and augmented if and only if it is isomorphic to a coideal subalgebra of $H$ (here, $\delta = \Delta|_A$).
\end{lemma}

\begin{proof}
Every left coideal subalgebra $A$ of $H$ is $H$-simple by \cite[Theorem~6.1]{Skryabin}, and is augmented by using the counit map of $H$ restricted to $A$. On the other hand, every $H$-simple and augmented left $H$-comodule algebra $A$ is isomorphic to a coideal subalgebra due to \cite[Lemma~4.5]{NSS2025pp}.
\end{proof}

%%%%%%%%%%%%%%%%%%%%%%%%%%
%%%%%%%%%%%%%%%%%%%%%%%%%%
%%%%%%%%%%%%%%%%%%%%%%%%%%

\section{Main results}\label{sec:results}
Here, we present the main results of this work on using braid groups of type BC  to produce invariants of (braided) Morita equivalence classes of (braided) module categories [Section~\ref{sec:typeBC}]. This is followed by similar results using braid groups of type D [Section~\ref{sec:typeD}], and 
preceded with a discussion of prior results in the type A case pertaining to braided monoidal categories [Section~\ref{sec:typeA}].

\smallskip

In the results below, let the subscript ``\textnormal{reg}" attached to an algebra denote its regular module.
%%%%%%%%%%%%%%%%%%%%%%%%%%

\subsection{Prior results in Coxeter type A}  \label{sec:typeA}
Let us recall a well-known result; see, e.g., \cite[$\S$X.6.2]{Kassel}.

\begin{proposition} \label{prop:typeA}
Take a braided monoidal category $(\mathcal{C}, \otimes, \one, c)$ with an object $X \in \mathcal{C}$. For each $n \geq 2$, the following map is a representation of a braid group of type \textnormal{A}:
\[
\begin{array}{rl}
\rho_n^X: \textnormal{Br}_{n-1}^{\textnormal{A}} &\longrightarrow \Aut_{\mathcal{C}}(X^{\otimes n})\\[.2pc]
\sigma_i &\mapsto \ide_{X^{\otimes i-1}} \otimes c_{X,X} \otimes \ide_{X^{\otimes n-i-1}}.
\end{array}
\]

\vspace{-.3in}

\qed
\end{proposition}

Shimizu \cite{Shimizu2010} shows that such representations serve as invariants of (braided) Morita equivalence classes of finite-dimensional (quasitriangular) Hopf algebras.

\begin{theorem} \cite[Lemma~2.11 and Theorem~3.1]{Shimizu2010} \label{thm:typeA}
Retain the notation of Proposition~\ref{prop:typeA}.
\begin{enumerate}[\upshape (a)]
\item Take finite-dimensional quasitriangular Hopf algebras $H$ and $H'$. If $\lmod{H} \simeq \lmod{H'}$ as braided monoidal categories, then $\rho_n^{H_\textnormal{reg}} \simeq \rho_n^{H'_\textnormal{reg}}$ as representations of  $\textnormal{Br}_{n-1}^{\textnormal{A}}$, for each $n \geq 2$.

\smallskip

\item Take finite-dimensional Hopf algebras $L$ and $L'$. If $\lmod{L} \simeq \lmod{L'}$ as monoidal categories, then $\rho_n^{D(L)_\textnormal{reg}} \simeq \rho_n^{D(L')_\textnormal{reg}}$ as representations of  $\textnormal{Br}_{n-1}^{\textnormal{A}}$, for each $n \geq 2$. \qed
\end{enumerate}
\end{theorem}

We will expand part (a) of this result to the (braided) module-theoretic setting next.

%%%%%%%%%%%%%%%%%%%%%%%%%%

\subsection{Main results in Coxeter type BC} \label{sec:typeBC}  We begin with constructing  representations of braid groups of Type BC using objects of braided module categories. This result is somewhat well-known, but we provide a proof for the reader's convenience. First, let us recall some useful identities.

\begin{lemma} \label{lem:cece}
Take a braided monoidal category $(\mathcal{C}, \otimes, \one, c)$ and a braided left $\mathcal{C}$-module category $(\mathcal{M}, \act, e)$. Then, for each $X,Y, Z \in \mathcal{C}$ and $M \in \mathcal{M}$:
\smallskip
\begin{enumerate}[\upshape (a)]
\item $(c_{Y,Z} \otimes \ide_X)(\ide_Y \otimes c_{X,Z}) (c_{X,Y} \otimes \ide_Z) = (\ide_Z \otimes c_{X,Y})(c_{X,Z} \otimes \ide_Y)(\ide_X \otimes c_{Y,Z})$; 

\medskip

\item $(c_{Y,X} \act \ide_M)(\ide_Y \otimes e_{X,M})(c_{X,Y} \act \ide_M)(\ide_X \otimes e_{Y,M})\\[.2pc]
 = (\ide_X \otimes e_{Y,M})(c_{Y,X} \act \ide_M)(\ide_Y \otimes e_{X,M})(c_{X,Y} \act \ide_M)$.
\end{enumerate}
\end{lemma}

\begin{proof}
(a)  The left-hand side is equal to $c_{Y \otimes X, Z} \, (c_{X,Y} \otimes \ide_Z)$ by \eqref{eq:braid2}, which in turn is equal to  $(\ide_Z \otimes c_{X,Y})\, c_{X \otimes Y, Z}$ by the naturality of $c$. This output is the right-hand side again by \eqref{eq:braid2}.

\smallskip
(b) The left-hand side is equal to $e_{X,Y \act M} \, (\ide_X \otimes e_{Y,M})$  by \eqref{eq:brmod2}, which in turn is equal to  $(\ide_X \otimes e_{Y,M})\, e_{X,Y \act M}$  by the naturality of $e$. This output is the right-hand side again by \eqref{eq:brmod2}.
\end{proof}

\begin{proposition} \label{prop:typeBC}
Take a braided monoidal category $(\mathcal{C}, \otimes, \one, c)$ with an object $X \in \mathcal{C}$, and a braided left $\mathcal{C}$-module category $(\mathcal{M}, \act, e)$ with an object $M \in \mathcal{M}$. For each $n \geq 2$, the following map is a representation of a braid group of type \textnormal{BC}:
\[
\begin{array}{rl}
\rho_n^{X,M}: \textnormal{Br}_n^{\textnormal{BC}} &\longrightarrow \Aut_{\mathcal{C}}(X^{\otimes n} \act M)\\[.2pc]
\sigma_i &\mapsto \ide_{X^{\otimes i-1}} \otimes c_{X,X} \otimes \ide_{X^{\otimes n-i-1}\act M}\\[.2pc]
t &\mapsto \ide_{X^{\otimes n-1}} \otimes e_{X,M}.
\end{array}
\]
\end{proposition}

\begin{proof}
Let $\rho:=\rho_n^{X,M}$. By Lemma~\ref{lem:cece}(a) and by level exchange, respectively, we have that 
\[
\begin{array}{ll}
\rho(\sigma_i)\rho(\sigma_{i+1})\rho( \sigma_i) = \rho(\sigma_{i+1})\rho(\sigma_i )\rho(\sigma_{i+1}) \, &\text{ for $i=1, \dots, n-2$},\\[.2pc]
\rho(\sigma_i)\rho(\sigma_j) =  \rho(\sigma_j)\rho(\sigma_i) \, &\text{ for $|i-j| \geq 2$}.
\end{array}
\] 
Next, by level exchange and by Lemma~\ref{lem:cece}(b), respectively, we obtain that
\[
\rho(\sigma_i)\rho(t) =  \rho(t)\rho(\sigma_i) \, \text{ for $i = 1, \dots, n-2$}, \; \qquad \rho(\sigma_{n-1})\rho(t) \rho(\sigma_{n-1})\rho(t) = \rho(t)\rho(\sigma_{n-1})\rho(t)\rho(\sigma_{n-1}).
\]
Thus, $\rho$ is a group homomorphism.
\end{proof}

Next, we present a useful result due to Andruskiewitsch--Mombelli \cite{AndruskiewitschMombelli} and  Skryabin \cite{Skryabin}.

\begin{proposition} \label{prop:AM-S}
Take a finite-dimensional Hopf algebra $H$ with left $H$-comodule algebras  $A$~and~$A'$.
\begin{enumerate}[\upshape (a)]
\item If $F: \lmod{A} \to \lmod{A'}$ is an equivalence of left $(\lmod{H})$-module categories, then 
\[
F \; \cong \; P \otimes_A -,
\] 
for a $(A',A)$-bimodule $P$ that is equipped with an $(A',A)$-bimodule map $P \to H \otimes P$ serving as a left $H$-coaction map.

\smallskip

\item  
If, further, $A$ and $A'$ are $H$-simple and augmented, then $P \cong A$ as right $A$-modules and $P \cong A'$ as left $A'$-modules.

 \end{enumerate}
\end{proposition}

\begin{proof}
(a) This follows from \cite[Proposition~1.23]{AndruskiewitschMombelli}.

\smallskip

(b) 
Note that the quasi-inverse of $F$ in part (a) is naturally isomorphic to a functor $Q \otimes_{A'} -$, for a $(A,A')$-bimodule $Q$ that is equipped with an $(A,A')$-bimodule map $Q \to H \otimes Q$ serving as a left $H$-coaction map. Moreover, $P \otimes_A Q \cong A'$ and $Q \otimes_{A'} P \cong A$; see \cite[Proposition~1.24]{AndruskiewitschMombelli}. Now by the hypotheses on $A$ and $A'$, the kernel of their respective augmentation maps can be used in the application of \cite[Theorem~4.2]{Skryabin} to obtain that $P$ is a free right $A$-module and that $Q$ is a free right $A'$-module. Say $P \cong A^{\oplus n}$ as right $A$-modules, and $Q \cong A'^{\, \oplus m}$ as right $A'$-modules, for some $n,m \in \mathbb{N}_{\geq 1}$. Now, as vector spaces:
\[ 
A' \; \cong \; P \otimes_A Q \; \cong \; Q^{\oplus n} \; \cong \; A'^{\,\oplus nm}.
\]
Thus, $n = m = 1$, and $P \cong A$ as right $A$-modules. The left statement follows likewise.
\end{proof}

Now we present the main result of the article.

\begin{theorem} \label{thm:typeBC}
Retain the notation of Proposition~\ref{prop:typeBC}, and take a finite-dimensional quasitriangular Hopf algebra $H$.  Take quasitriangular left coideal subalgebras  $B$ and $B'$. If $\lmod{B} \simeq \lmod{B'}$ as braided left $(\lmod{H})$-module categories, then 
\[
\rho_n^{H_\textnormal{reg}, B_\textnormal{reg}} \; \simeq \; \rho_n^{H_\textnormal{reg}, B'_\textnormal{reg}}
\] 
as representations of  $\textnormal{Br}_n^{\textnormal{BC}}$, for each $n \geq 2$.
\end{theorem}

\begin{proof}
We aim to produce a linear isomorphism  $T: H^{\otimes n} \otimes B \to H^{\otimes n} \otimes B'$ such that if there is equivalence of braided left $(\lmod{H})$-module categories, $F: \lmod{B} \to \lmod{B'}$, then the  diagram below commutes:
\medskip
{\small\[ 
\xymatrix@R=0.8pc{
H^{\otimes n} \otimes B 
\ar@/^1.3pc/[rrrrrr]_{T}
\ar[rr]_{T_1}
\ar[dd]_{\rho_n^{H_\textnormal{reg},B_\textnormal{reg}}(\theta)}
&& 
F(H^{\otimes n} \otimes B)
\ar[rr]_{T_2}
\ar[dd]_{F(\rho_n^{H_\textnormal{reg},B_\textnormal{reg}}(\theta))}
&&
H^{\otimes n} \otimes F(B)
\ar[rr]_{T_3}
\ar[dd]^{\rho_n^{H_\textnormal{reg},F(B_\textnormal{reg})}(\theta)}
&&
H^{\otimes n} \otimes B'
\ar[dd]^{\rho_n^{H_\textnormal{reg},B'_\textnormal{reg}}(\theta)}\\
&&&&&&\\
H^{\otimes n} \otimes B
\ar@/_1.3pc/[rrrrrr]^{T} 
\ar[rr]^{T_1}
&& 
F(H^{\otimes n} \otimes B)
\ar[rr]^{T_2}
&&
H^{\otimes n} \otimes F(B)
\ar[rr]^{T_3}
&&
H^{\otimes n} \otimes B'
}
\]}

\medskip

\noindent for all $\theta \in \textnormal{Br}_n^{\textnormal{BC}}$. Here, $H$ and $B$ are the underlying vector spaces of $H_{\textnormal{reg}}$ and $B_{\textnormal{reg}}$, respectively. This will be achieved by producing the intermediate linear maps $T_1$, $T_2$, $T_3$ pictured above.

\smallskip

To define $T_1$, consider the forgetful functors to the category of vector spaces, $U: \lmod{B} \to \mathsf{Vec}$ and $U': \lmod{B'} \to \mathsf{Vec}$. Then, there exists a natural isomorphism 
$\Psi: U \natisom  U'  F$.
Indeed, by applying Lemma~\ref{lem:HsimpleCSA} and Proposition~\ref{prop:AM-S}, we get that $F \cong P \otimes_B -$, for a $(B',B)$-bimodule $P$ such that there is an isomorphism $\phi: P \equivto B$ of right $B$-modules. Now for a  left $B$-module $(V, \cdot)$, we get that $U(V) = V$ and $U'F(V) \cong P \otimes_B V$, as vector spaces. It is straight-forward to check that the maps
\[
\Psi_V: V \equivto P \otimes_B V \;\; (\cong F(V)), \quad v \mapsto \phi^{-1}(1_B) \otimes_B v
\] 
form the components of the natural isomorphism $\Psi$ above. Here, the inverse of $\Psi_V$ is given by the map $P \otimes_B V \to V$ defined by $p \otimes_B v \mapsto \phi(p) \cdot v$. Now take $T_1:= \Psi_{H^{\otimes n} \otimes B}$. We then obtain that the left region of the diagram above commutes by the naturality of $\Psi$, for any choice of $\theta \in  \textnormal{Br}_n^{\textnormal{BC}}$.

\smallskip

Define $T_2:= s_{H^{\otimes n}, B}: F(H^{\otimes n} \otimes B) \equivto H^{\otimes n} \otimes F(B)$, the module functor constraint for $F$. Then, for the element $\theta = \sigma_i$ of $\textnormal{Br}_n^{\textnormal{BC}}$, we get: 
\begin{align*}
\rho_n^{H_\textnormal{reg},F(B_\textnormal{reg})}(\sigma_i)  \circ s_{H^{\otimes n}, B} 
& \; = \; \left(\ide_{H^{\otimes i-1}} \otimes c_{H,H} \otimes \ide_{H^{\otimes n-i-1}\act F(B)}\right) \circ s_{H^{\otimes n}, B}\\
& \overset{\textnormal{$s$\,nat'l}}{=} s_{H^{\otimes n}, B}  \circ F(\ide_{H^{\otimes i-1}} \otimes c_{H,H} \otimes \ide_{H^{\otimes n-i-1}\act B})  \\
& \; = \; s_{H^{\otimes n}, B} \circ  F \hspace{-.03in} \left(\rho_n^{H_\textnormal{reg},B_\textnormal{reg}}(\sigma_i)\right).
\end{align*}
Next, for the element $\theta = t$ of $\textnormal{Br}_n^{\textnormal{BC}}$, we get:
\begin{align*}
\rho_n^{H_\textnormal{reg},F(B_\textnormal{reg})}(t)  \circ s_{H^{\otimes n}, B} 
& \; = \; \left(\ide_{H^{\otimes n-1}} \otimes e_{H,F(B)}\right) \circ s_{H^{\otimes n}, B}\\
& \overset{\textnormal{pent.\,axiom for $s$}}{=} \left(\ide_{H^{\otimes n-1}} \otimes e_{H,F(B)}\right)\circ (\ide_{H^{\otimes n-1}} \otimes s_{H,B}) \circ s_{H^{\otimes n-1}, H \otimes B} \\
& \overset{\eqref{eq:brmodfun}}{=} (\ide_{H^{\otimes n-1}} \otimes s_{H,B}) \circ \left(\ide_{H^{\otimes n-1}} \otimes F(e_{H,B})\right)\circ s_{H^{\otimes n-1}, H \otimes B} \\
& \overset{\textnormal{$s$\,nat'l}}{=} (\ide_{H^{\otimes n-1}} \otimes s_{H,B})\circ s_{H^{\otimes n-1}, H \otimes B} \circ F \hspace{-.03in} \left(\ide_{H^{\otimes n-1}} \otimes e_{H,B}\right) \\
& \overset{\textnormal{pent.\,axiom for $s$}}{=}  s_{H^{\otimes n}, B} \circ  F \hspace{-.03in} \left(\rho_n^{H_\textnormal{reg},B_\textnormal{reg}}(t)\right).
\end{align*}
Thus, the middle region of the diagram above commutes with the choice of $T_2$ above.

\smallskip

Finally, we define $T_3$. By Lemma~\ref{lem:HsimpleCSA} and  Proposition~\ref{prop:AM-S},  $F \cong P \otimes_B -$, for a $(B',B)$-bimodule $P$ such that $P \cong B'$ as left $B'$-modules. So, we get isomorphisms $F(B) \cong P \otimes_B B \cong P \cong B'$ as left $B'$-modules, and in particular, as vector spaces. Label this linear isomorphism by $\phi'$. Now take $T_3:= \ide_{H^{\otimes n}} \otimes \phi'$. Then, for the element $\theta = \sigma_i$ (resp., for $\theta = t$) of $\textnormal{Br}_n^{\textnormal{BC}}$, the right region of the diagram above commutes by level exchange (resp., by the naturality of $e$).

\smallskip

Thus, we have defined a linear isomorphism $T := T_3\, T_2\, T_1: H^{\otimes n} \otimes B \equivto H^{\otimes n} \otimes B'$ that yields an equivalence $\rho_n^{H_\textnormal{reg}, B_\textnormal{reg}} \simeq \rho_n^{H_\textnormal{reg}, B'_\textnormal{reg}}$ 
of representations of  $\textnormal{Br}_n^{\textnormal{BC}}$, for each $n \geq 2$.
\end{proof}

\begin{corollary} \label{cor:csa-dim-same}
Take a finite-dimensional quasitriangular Hopf algebra $H$, with quasitriangular left coideal subalgebras $B$ and $B'$. If $\lmod{B} \simeq \lmod{B'}$ as braided left $(\lmod{H})$-module categories, then 
\[
\textnormal{dim}_\Bbbk B \; = \; \textnormal{dim}_\Bbbk B'.
\]
\end{corollary}

\begin{proof}
 Apply Theorem~\ref{thm:typeBC} to obtain that 
\[ \textstyle
\text{dim}_\Bbbk B \; = \; \frac{1}{(\text{dim}_\Bbbk H)^{n}} \, \text{trace}\big(\rho_n^{H_{\text{reg}}, B_{\text{reg}}}(1_{\textnormal{Br}_n^{\textnormal{BC}}})\big)\; =\; \frac{1}{(\text{dim}_\Bbbk H)^{n}}\,  \text{trace}\big(\rho_n^{H_{\text{reg}}, B'_{\text{reg}}}(1_{\textnormal{Br}_n^{\textnormal{BC}}})\big)\; = \;\text{dim}_\Bbbk B'.
\]

\vspace{-.2in}

\end{proof}

\vspace{.1in}

As a direction for further research, we ask:

\begin{question} \label{ques:hyp}
Can one obtain a generalization of Theorem~\ref{thm:typeBC} for $H$-comodule algebras more general than coideal subalgebras of $H$? For instance, can one obtain an analogue of Theorem~\ref{thm:typeA}(b) using reflective algebras in place of Drinfeld doubles?
\end{question}

\begin{question} \label{ques:RHA}
For finite-dimensional Hopf algebras $L$ and $L'$, we have that $\lmod{L}$ and $\lmod{L'}$ are {\it categorically Morita equivalent}  if and only if the Drinfeld doubles $D(L)$ and $D(L')$ are braided Morita equivalent \cite[Definition~7.12.11, Theorem~8.12.3]{EGNO}. Likewise, is there a notion of categorical Morita equivalence for categories of left modules over comodule algebras that will be tied to the braided Morita equivalence of their reflective algebras?
\end{question}

%%%%%%%%%%%%%%%%%%%%%%%%%%

\subsection{Results in Coxeter type D} \label{sec:typeD} We adapt the results above for type D braid groups. We first produce representations of such groups using objects of symmetric module categories.

\begin{proposition} \label{prop:typeD}
Take a braided monoidal category $(\mathcal{C}, \otimes, \one, c)$ with an object $X \in \mathcal{C}$, and a symmetric left $\mathcal{C}$-module category $(\mathcal{M}, \act, e)$ with an object $M \in \mathcal{M}$. For each $n \geq 2$, the following map is a representation of a braid group of type \textnormal{D}:
\[
\begin{array}{rl}
\rho_n^{X,M}: \textnormal{Br}_n^{\textnormal{D}} &\longrightarrow \Aut_{\mathcal{C}}(X^{\otimes n} \act M)\\[.2pc]
\sigma_i &\mapsto \ide_{X^{\otimes i-1}} \otimes c_{X,X} \otimes \ide_{X^{\otimes n-i-1}\act M}\\[.2pc]
t &\mapsto \ide_{X^{\otimes n-2}} \otimes (\ide_X \otimes e_{X,M})(c_{X,X} \act \ide_M)(\ide_X \otimes e_{X,M}).
\end{array}
\]
\end{proposition}

\begin{proof}
Let $\rho:=\rho_n^{X,M}$.  By Lemma~\ref{lem:cece}(a) and by level exchange, respectively, we have that 
\[
\begin{array}{ll}
\rho(\sigma_i)\rho(\sigma_{i+1})\rho(\sigma_i) = \rho(\sigma_{i+1})\rho(\sigma_i)\rho(\sigma_{i+1}) \, &\text{ for $i=1, \dots, n-2$}, \\[.2pc]
\rho(\sigma_i)\rho(\sigma_j) =  \rho(\sigma_j)\rho(\sigma_i) \, &\text{ for $|i-j| \geq 2$}.
\end{array}
\]
Next, by level exchange  and by Lemma~\ref{lem:cece}(b), respectively, we obtain that 
\[
\rho(\sigma_i)\rho(t) =  \rho(t)\rho(\sigma_i) \, \text{ for $i = 1, \dots, n-3$}, \qquad  \;
\rho(\sigma_{n-1})\rho(t) = \rho(t)\rho(\sigma_{n-1}).
\] 
Finally, the proof that $\rho$ is a group homomorphism is concluded by the computation below: 
{\small
\begin{align*}
& \rho(\sigma_{n-2})\rho(t)\rho(\sigma_{n-2}) \\
& =  \ide_{X^{\otimes n-3}} \otimes (c_{X,X} \otimes \ide_{X \act M})(\ide_{X \otimes X} \otimes e_{X,M})(\ide_X \otimes c_{X,X} \act \ide_M)(\ide_{X \otimes X} \otimes e_{X,M}) (c_{X,X} \otimes \ide_{X \act M})\\
& \overset{\textnormal{level ex.}}{=}  \ide_{X^{\otimes n-3}} \otimes (\ide_{X \otimes X} \otimes e_{X,M})(c_{X,X} \otimes \ide_{X \act M})(\ide_X \otimes c_{X,X} \act \ide_M) (c_{X,X} \otimes \ide_{X \act M}) (\ide_{X \otimes X} \otimes e_{X,M})\\
& \overset{\textnormal{Lem.\,\ref{lem:cece}(a)}}{=}  \ide_{X^{\otimes n-3}} \otimes (\ide_{X \otimes X} \otimes e_{X,M})(\ide_X \otimes c_{X,X} \act \ide_M) (c_{X,X} \otimes \ide_{X \act M})(\ide_X \otimes c_{X,X} \act \ide_M)\\
&\hspace{1.1in} \circ (\ide_{X \otimes X} \otimes e_{X,M})\\
& \overset{e^2=\ide}{=}  \ide_{X^{\otimes n-3}} \otimes (\ide_{X \otimes X} \otimes e_{X,M})(\ide_X \otimes c_{X,X} \act \ide_M) (c_{X,X} \otimes \ide_{X \act M})(\ide_{X \otimes X} \otimes e_{X,M}) \\
&\hspace{1.1in} \circ (\ide_{X \otimes X} \otimes e_{X,M})(\ide_X \otimes c_{X,X} \act \ide_M)(\ide_{X \otimes X} \otimes e_{X,M})\\
&  \overset{\textnormal{level ex.}}{=}  \ide_{X^{\otimes n-3}} \otimes (\ide_{X \otimes X} \otimes e_{X,M})(\ide_X \otimes c_{X,X} \act \ide_M)(\ide_{X \otimes X} \otimes e_{X,M}) (c_{X,X} \otimes \ide_{X \act M}) \\
&\hspace{1.1in} \circ (\ide_{X \otimes X} \otimes e_{X,M})(\ide_X \otimes c_{X,X} \act \ide_M)(\ide_{X \otimes X} \otimes e_{X,M})\\
& = \rho(t)\rho(\sigma_{n-2})\rho(t). 
\end{align*}
}

\vspace{-.2in}

\end{proof}

\begin{remark} \label{rem:allcock}
The idea for Proposition~\ref{prop:typeD} was motivated by the work of Allcock on braid pictures for Artin groups \cite{Allcock}. In particular, the symmetric condition on $\mathcal{M}$ corresponds to the double application of the ``orbifold move" on \cite[page~3459]{Allcock}.
\end{remark}

Now we present the main result of this part; cf. Theorem~\ref{thm:typeBC}.

\begin{theorem} \label{thm:typeD}
Retain the notation of Proposition~\ref{prop:typeD}, and take a finite-dimensional quasitriangular Hopf algebra $H$. Take triangular left coideal subalgebras   $B$ and $B'$. If $\lmod{B} \simeq \lmod{B'}$ as braided left $(\lmod{H})$-module categories, then 
\[
\rho_n^{H_\textnormal{reg}, B_\textnormal{reg}} 
\; \simeq \; \rho_n^{H_\textnormal{reg}, B'_\textnormal{reg}}
\] 
as representations of  $\textnormal{Br}_n^{\textnormal{D}}$, for each $n \geq 2$.
\end{theorem}

\begin{proof}
This proof follows similarly to the proof of Theorem~\ref{thm:typeBC}. Recall the diagram in the proof of Theorem~\ref{thm:typeBC}. The linear map $T_1$ (resp.,  $T_3$) is produced in exactly the same manner as in Theorem~\ref{thm:typeBC}, and it makes the left (resp., right) region of this diagram commute. Finally, we choose $T_2:= s_{H^{\otimes n}, B}$ as in Theorem~\ref{thm:typeBC}. We get that $\rho_n^{H_\textnormal{reg},B_\textnormal{reg}}(\sigma_i) \circ T_2 = T_2 \circ \rho_n^{H_\textnormal{reg},B_\textnormal{reg}}(\sigma_i)$ as in the type BC case. We also obtain that  $\rho_n^{H_\textnormal{reg},B_\textnormal{reg}}(t) \circ T_2 = T_2 \circ \rho_n^{H_\textnormal{reg},B_\textnormal{reg}}(t)$ here by employing the argument in the type BC case and the naturality of $s$.
Thus, we have linear isomorphism $T := T_3\, T_2\, T_1: H^{\otimes n} \otimes B \equivto H^{\otimes n} \otimes B'$ that yields an equivalence $\rho_n^{H_\textnormal{reg}, B_\textnormal{reg}} \simeq \rho_n^{H_\textnormal{reg}, B'_\textnormal{reg}}$ 
of representations of  $\textnormal{Br}_n^{\textnormal{D}}$, for each $n \geq 2$.
\end{proof}

Given the results in this part, and in the previous part,  we inquire:

\begin{question} \label{ques:othertypes}
Are there analogues of Propositions~\ref{prop:typeA}, \ref{prop:typeBC}, and \ref{prop:typeD} for braid groups of other Coxeter types? If so, what type of braided categorical structures are involved? Further,  are there  analogues of Theorems~\ref{thm:typeA}, \ref{thm:typeBC}, and \ref{thm:typeD} in other Coxeter types?
\end{question}

%%%%%%%%%%%%%%%%%%%%%%%%%%
%%%%%%%%%%%%%%%%%%%%%%%%%%
%%%%%%%%%%%%%%%%%%%%%%%%%%

\section{Examples}\label{sec:examples} 

In this section, we classify $K$-matrices  and determine the braided Morita equivalence classes of coideal subalgebras of certain finite-dimensional quasitriangular Hopf algebras $H$.  Here, $H$ will be a finite group algebra in Section~\ref{ss:group}, and the Sweedler algebra in Section~\ref{ss:H4}. 

\smallskip

To proceed, consider the preliminary result below, which refines Proposition~\ref{prop:AM-S}(a).

\begin{proposition}\label{prop:equivalence}
Let $(H, R)$ be a  finite-dimensional  quasitriangular Hopf algebra, and let \linebreak $(B, \, K \! = \! \sum_i K_i\otimes K^i)$, $(B', \, K' \! = \!\sum_j K'_j\otimes K'^j)$ be quasitriangular left $H$-comodule algebras. Then, 
\[
\lmod{B} \; \simeq  \; \lmod{B'}
\]
as braided left $(\lmod{H})$-module categories if and only if each of the conditions below hold.
\begin{enumerate}[\upshape (i)]
\item There exists a $(B', B)$-bimodule $P$ equipped with a $(B',B)$-bimodule map  
\[
\lambda: P\to H\otimes P, \quad p \mapsto p_{\langle -1 \rangle} \otimes p_{\langle 0 \rangle}
\]
serving as a left $H$-coaction map.
\item There exists a $(B, B')$-bimodule $Q$ such that $P\otimes_{B} Q \cong B'$ as $B'$-bimodules and $Q\otimes_{B'} P \cong B$ as $B$-bimodules (i.e., $P$ is invertible). 
\item For all $(X, \cdot) \in \lmod{H}$ and $(M, \ast) \in\lmod{B}$, with $p\in P$, $x \in X$, $m\in M$, we have that:
\[
\textstyle \sum_i (p_{\langle -1 \rangle} \, K_i\cdot x) \, \otimes \,p_{\langle 0 \rangle} \,\otimes_{B} \,(K^i \ast m) \; = \;  \sum_j (K'_j \, p_{\langle -1 \rangle}\cdot x) \,\otimes \,(K'^j \, p_{\langle 0 \rangle})\, \otimes_{B} \, m.
\]
\end{enumerate}
 
\end{proposition}

\begin{proof}
Items (i) and (ii) provide necessary and sufficient conditions for $\lmod{B}  \simeq  \lmod{B'}$
as left $(\lmod{H})$-module categories by \cite[Proposition~1.23]{AndruskiewitschMombelli}. In particular, the module functor $(F, s):\lmod{B} \to \lmod{B'}$ is given by $F=P\otimes_B-$, and for all $(X, \cdot) \in \lmod{H}$ and $(M, \ast) \in\lmod{B}$, the module functor constraint is given by
\[
s_{X, M}: P\otimes_B X\otimes M\to X\otimes P\otimes_B M, \quad p \, \otimes_B   x  \otimes  m\mapsto (p_{\langle -1 \rangle}\cdot x) \, \otimes  \, p_{\langle 0 \rangle} \, \otimes_B \, m.
\]
Now item (iii) is a necessary and sufficient condition for this equivalence to be of braided module categories. Here, the braidings of $\lmod{B}$ and $\lmod{B'}$ are given by the respective $K$-matrices as described by \eqref{eq:formula-e}. With this, condition \eqref{eq:brmodfun} is equivalent to item~(iii). 
\end{proof}

%%%%%%%%%%%%%%%%%%%%%%%%%%

\subsection{With group algebras}  \label{ss:group} 
Let $G$ be a finite group with a central element $u \in Z(G)$ with $u^2 = 1$. The group algebra $\Bbbk G$ is a finite-dimensional quasitriangular Hopf algebra with $R$-matrix:
\begin{equation} \label{eq:Ru}
\textstyle R_u=\frac{1}{2}(1 \otimes 1 + 1 \otimes u + u \otimes 1 - u \otimes u).
\end{equation}
Also, $(R_u)_{21} = (R_u)^{-1}$.  The left coideal subalgebras of $\Bbbk G$ are $\Bbbk L$, for $L$ a subgroup of $G$. Retain this notation for the rest of this section.

\begin{proposition} \label{prop:kG-K}
    The $K$-matrices of $\Bbbk L$ are $K=a\otimes 1$, where $a\in C_G(L)$ (centralizer of $L$ in $G$).
\end{proposition}
\begin{proof} Take $\widehat{K}=\sum_{g\in G}\alpha_g g\in \Bbbk G$, for $\alpha_g\in \Bbbk$. By \eqref{eq:K-matrix-coideal-sub}$(\widehat{\text{i}})$ and a routine calculation, 
\[
  \textstyle  \sum_{g\in G} \, \alpha_{g} \, g\otimes g  \; = \; \sum_{g', g''\in G} \alpha_{g'} \alpha_{g''} \, g'\otimes g''.
\] So, $\alpha_{g}\in\{0,1\}$, for all $g\in G$, and $\alpha_{g'}\alpha_{g''}=0$ if $g'\neq g''$. Since $K$ is invertible, there exist only one $\alpha_{a}\neq 0$, for some $a\in G$. Then, $\widehat{K}=a$, for some $a\in G$. By  \eqref{eq:K-matrix-coideal-sub}$(\widehat{\text{iii}})$,  $a\ell=\ell a$, for all $\ell\in L$. Therefore, $a\in C_G(L)$. Note that \eqref{eq:K-matrix-coideal-sub}$(\widehat{\text{ii}})$ is satisfied. By Lemma \ref{lemma:criteria-K-matrix}, $K=a\otimes 1$, where $a\in C_G(L)$.
\end{proof}

In the next result, we show how to distinguish the braided Morita equivalence classes of the quasitriangular coideal subalgebras $(\Bbbk L, a \otimes 1)$, for $a \in C_G(L)$, above.

\begin{theorem} \label{thm:kG}
Given the quasitriangular Hopf algebra $(\Bbbk G, R_u)$, 
its quasitriangular coideal subalgebras $(\Bbbk L, K=a\otimes 1)$ for $a \in C_G(L)$ and $(\Bbbk L', K'=a'\otimes 1)$ for $a' \in C_G(L')$ are braided Morita equivalent if and only if $L'=L^g :=\{g\ell g^{-1}, \ell\in L\}$ and $K'=K^g := gag^{-1}\otimes 1$, for some $g\in G$.
\end{theorem} 

\begin{proof}
Suppose that $L'=L^g$ and $K'=K^g$, for some $g\in G$. Since $L'=L^g$, we get that $\lmod{\Bbbk L} \simeq \lmod{\Bbbk L'}$ as left $(\lmod{\Bbbk G})$-module categories; see \cite[Theorem~3.1]{Ostrik-IMRN} and also \cite[Exercise~7.4.11, Remark~7.4.12]{EGNO}. Let $P=\Bbbk (L^ggL)$ be a $(\Bbbk L^g, \Bbbk L)$-bimodule with multiplication as the left and right action, and comultiplication as the left $\Bbbk G$-coaction. Likewise, take the $(\Bbbk L, \Bbbk L^g)$-bimodule $Q=\Bbbk(Lg^{-1}L^g)$. Then, by using \cite[Lemma~4.5.2]{Cohn} (see also \cite[Proposition~2.20]{Morales-et-al}), we have the isomorphisms: 
\[
\begin{array}{ll}
P\otimes_{\Bbbk L} Q \cong \Bbbk L^g, &g\ell_1\ell_2\otimes_{\Bbbk L} \ell'_1\ell'_2g^{-1}\mapsto g\ell_1\ell_2\ell'_1\ell'_2g^{-1},\\[.4pc]
Q\otimes_{\Bbbk L^g} P \cong \Bbbk L, 
&\ell_1\ell_2g^{-1}\otimes_{\Bbbk L^g} g\ell'_1\ell'_2\mapsto \ell_1\ell_2\ell'_1\ell'_2,
\end{array}
\]
as $\Bbbk L^g$-bimodules and $\Bbbk L$-bimodules, respectively.  Thus, we have items (i) and (ii) of Proposition~\ref{prop:equivalence}. Item (iii) holds as for all $X\in\lmod{H}$, $M\in\lmod{B}$, with $x\in X$, $m\in M$, $p=g\ell_1\ell_2\in P$,
\[
g\ell_1\ell_2ax \, \otimes \, g\ell_1\ell_2 \, \otimes_{\Bbbk L} \, m  \overset{a\in C_G(L)}{=} ga\ell_1\ell_2x \,\otimes \, g\ell_1\ell_2 \, \otimes_{\Bbbk L} \, m \overset{a'=gag^{-1}}{=} a'g \ell_1\ell_2x \,\otimes \,g\ell_1\ell_2 \,\otimes_{\Bbbk L} \,m.
\] 
Therefore, by Proposition \ref{prop:equivalence}, $(\Bbbk L, a\otimes 1)$ is braided Morita equivalent to $(\Bbbk L', a'\otimes 1)$.

Conversely, since $\lmod{\Bbbk L} \simeq \lmod{\Bbbk L'}$ as left $(\lmod{\Bbbk G})$-module categories, we get that $L'=L^g$, for some $g\in G$; see \cite[Theorem~3.1]{Ostrik-IMRN} and also \cite[Exercise~7.4.11, Remark~7.4.12]{EGNO}. (This is consistent with Corollary~\ref{cor:csa-dim-same}.) Now using the equivalence of braided module categories,  take the corresponding  $(\Bbbk L^g, \Bbbk L)$-bimodule $P$ from Proposition~\ref{prop:equivalence}. Here, $P \cong \Bbbk (\bigoplus_{i=1}^n L^g h_i L)$ for some $h_i \in G$. Since $P \cong \Bbbk L$ as right $\Bbbk L$-modules by Proposition~\ref{prop:AM-S}(b), and thus as vector spaces, $P \cong L^g h L$ for some $h \in G$. Now $|L| = |L^g h L| = |L^g| \, |L|\,  / \, |L^g \cap L^h|$. So, $h = g$ and $P \cong L^g g L$ as $(\Bbbk L^g, \Bbbk L)$-bimodules.
By item (iii) of Proposition~\ref{prop:equivalence}, we obtain that
\[
(p_{\langle -1 \rangle} \, a \cdot x) \otimes p_{\langle 0 \rangle} \otimes_B m \; = \; 
(a' \, p_{\langle -1 \rangle}  \cdot x) \otimes p_{\langle 0 \rangle} \otimes_B m.
\]
for all $p \in P$, $x \in X$ for $(X, \cdot) \in \lmod{\Bbbk G}$, and $m \in M$ for $(M, \ast) \in \lmod{\Bbbk L}$.
Choosing $p=g$ via the isomorphism above, 
with $m=1$ and $x=1$, we obtain that $a'=gag^{-1}$. Therefore, $K' = K^g$.
\end{proof}

\begin{example}
\begin{enumerate}[\upshape (a)]
\item If $G$ is abelian, then the braided Morita equivalence classes of quasitriangular coideal subalgebras of $\Bbbk G$ are singletons of the form $\{(\Bbbk L, \, K = g \otimes 1) \}_{L \leq G, \, g \in L}$. 

\smallskip

\item Take the group $S_3$, which has subgroups $L_\ast:=\langle (\ast) \rangle$, for $\ast = 1, 12, 13, 23, 123$, and $S_3$. The braided Morita equivalence classes of quasitriangular coideal subalgebras of $\Bbbk S_3$ are: 
{\small
\begin{itemize}
\smallskip
\item  $\{(L_1, 1 \otimes 1)\}$, 
\smallskip
\item  $\{(L_1, (12) \otimes 1), \, (L_1, (13) \otimes 1), \,  (L_1, (23) \otimes 1)\}$, 
\smallskip
\item  $\{(L_1, (123) \otimes 1), \,  (L_1, (132) \otimes 1)\}$,
\smallskip
\item  $\{(L_{12}, 1 \otimes 1), \, (L_{13}, 1 \otimes 1), \,  (L_{23}, 1 \otimes 1)\}$, 
\smallskip

\item  $\{(L_{12}, (12) \otimes 1), \, (L_{13}, (13) \otimes 1), \,  (L_{23}, (23) \otimes 1)\}$,
\smallskip
\item $\{(L_{123}, 1 \otimes 1)\}$, 
\smallskip
\item $\{(L_{123}, (123) \otimes 1), \,  (L_{123}, (132) \otimes 1)\}$, 
\smallskip
\item $\{(S_3, 1 \otimes 1)\}$.
\end{itemize}
}
\end{enumerate}
\end{example}

%%%%%%%%%%%%%%%%%%%%%%%%%%

\subsection{With the Sweedler algebra} \label{ss:H4} Let $H_4$ be the Sweedler Hopf algebra. As an algebra $H_4$ is generated by $g$ and $x$ subject to the relations $g^2=1$, $x^2=0$ and $xg=-gx$. The coalgebra structure is determined by $\Delta(g)=g\otimes g$ and $\Delta(x)=x\otimes g+1\otimes x$. The quasitriangular structures on $H_4$ are given by the $R$-matrices:
\begin{equation}\label{eq:Rlambda}
\textstyle R_\lambda=\frac{1}{2}(1\otimes 1+1\otimes g+g\otimes 1-g\otimes g)+\frac{\lambda}{2}(x\otimes x+x\otimes gx+gx\otimes gx-gx\otimes x), \quad \lambda\in\Bbbk,
\end{equation}
and $(R_{\lambda})_{21} = (R_\lambda)^{-1}$. See \cite[$\S$7.3, Exercise 12.2.11]{Radford1994}. The left coideal subalgebras of $H_4$ are $\Bbbk$,  $\Bbbk 1 \oplus \Bbbk gx$, $\Bbbk 1 \oplus \Bbbk (g+\tau gx)$ for $\tau\in \Bbbk$,  and $H_4$; see \cite[$\S$4]{CKS}. Retain this notation for the rest of this section.

\begin{proposition} \label{prop:H4-K}
The $K$-matrices for coideal subalgebras $B$ of $H_4$ are classified as follows.
\begin{enumerate}[\upshape (a)]
    \item For $B$ = $\Bbbk$, we have that $K\in \{1\otimes 1, \, g\otimes 1\}$ when $\lambda = 0$, and $K=1\otimes 1$ when $\lambda \neq 0$.
    \smallskip
     \item For $B = \Bbbk 1 \oplus \Bbbk (g+\tau gx)$, we have that $K\in \{1\otimes 1, \, g\otimes 1\}$ when $\lambda = \tau = 0$, and $K=1\otimes 1$ otherwise.
    \smallskip
    \item  The only $K$-matrix of $\Bbbk 1 \oplus \Bbbk gx$ is $K=1\otimes 1$.
       \smallskip 
    \item The only $K$-matrix of $H_4$ is $K=1\otimes 1$.
\end{enumerate}
\end{proposition}

\begin{proof} Let us use Lemma \ref{lemma:criteria-K-matrix}. Take $\widehat{K}=\alpha 1+\beta g+\gamma x+\delta gx\in H_4$. By a routine calculation, 
\[
(R_\lambda)_{21}\widehat{K}_1(R_\lambda)_{12}=\alpha\, 1\otimes 1+ \beta\, g\otimes 1+\gamma\,x\otimes g+\delta\, gx\otimes g +\beta\lambda\, gx\otimes gx+\beta\lambda\,  gx\otimes x. 
\]

For part (a), take the coideal subalgebra $B=\Bbbk$. By \eqref{eq:K-matrix-coideal-sub}$(\widehat{\text{ii}})$, $\gamma=\delta=\beta\lambda=0$. Then, $\widehat{K}=\alpha 1+\beta g$, where $\beta\lambda=0$. Note that \eqref{eq:K-matrix-coideal-sub}$(\widehat{\text{iii}})$ is satisfied. Finally, by \eqref{eq:K-matrix-coideal-sub}$(\widehat{\text{i}})$, 
\[\alpha\, 1\otimes 1+\beta\, g\otimes g \, =\, \alpha^2\,1\otimes 1+\alpha\beta\,g\otimes 1+\alpha\beta\, 1\otimes g+\beta^2\, g\otimes g.\] 
So $\alpha^2=\alpha$, $\alpha\beta=0$, and $\beta^2=\beta$. Since $\widehat{K}$ is invertible, the possible cases are $\beta=0$ and $\alpha=1$, or $\alpha=0$ and $\beta=1$. In the first case, we have that $\widehat{K}=1$, and thus $K=1\otimes 1$. In the second case, $\lambda=0$ and $\widehat{K}=g$; thus, $K=g\otimes 1$ when $\lambda=0$. 

For part (b), take the coideal subalgebra $B=\Bbbk 1 \oplus \Bbbk (g+\tau gx)$ for any $\tau \in \Bbbk$. By \eqref{eq:K-matrix-coideal-sub}$(\widehat{\text{iii}})$ with $b=g+\tau gx$, a routine calculation yields
\[
-\gamma\, gx+(-\delta+\beta\tau)x \, = \, \gamma\, gx+(\delta -\beta\tau)x.
\] 
So, $\gamma=0$ and $\delta=\beta\tau$. Then, $\widehat{K}=\alpha 1+\beta g+\beta\tau gx$. By \eqref{eq:K-matrix-coideal-sub}$(\widehat{\text{i}})$, 
\begin{align*}
\alpha\, 1\otimes 1+\beta g\otimes g+\beta\tau\, (gx\otimes 1+g\otimes gx)  = &
\; \alpha^2\, 1\otimes 1+\alpha \beta\tau\, gx\otimes g+ \beta\lambda(\alpha+\beta)gx\otimes gx+\alpha\beta\, g\otimes 1\\
& +\beta(\alpha\lambda+\beta\lambda-\beta\tau^2)gx\otimes x+\alpha\beta\,1\otimes g\\
& +\beta^2\tau\, gx\otimes 1+\beta^2\,g\otimes g+\beta\tau\alpha\, 1\otimes gx+\beta^2\tau\,g\otimes gx.
\end{align*} 
So, $\alpha^2=\alpha$, $\beta^2=\beta$, $\beta \tau^2=0$, $\beta\lambda=0$, and $\alpha\beta=0$. Since $\widehat{K}$ is invertible, $(\alpha, \beta) = (0,1)$ or $(1,0)$. If $\alpha=0$ and $\beta=1$, then $\lambda=\tau=0$, and $\widehat{K}=g$. Therefore, $K=g\otimes 1$, for $\lambda=\tau=0$. If $\alpha=1$ and $\beta=0$, then $\widehat{K}=1$. So $K=1\otimes 1$. Note that in either case the condition \eqref{eq:K-matrix-coideal-sub}$(\widehat{\text{ii}})$ is satisfied.

The proofs of parts (c) and (d) follow likewise.
\end{proof}

Next, we obtain the braided Morita classes of quasitriangular coideal subalgebras here.

\begin{theorem} \label{thm:H4}
The braided Morita equivalence classes of quasitriangular coideal subalgebras of $(H_4, R_\lambda)$ are classified as follows. When $\lambda = 0$, the classes are
\[
\{(\Bbbk, 1\otimes 1)\},\, \{(\Bbbk, g\otimes 1)\},\; \{(\Bbbk 1\oplus \Bbbk g, 1\otimes 1)\},\; \{(\Bbbk 1\oplus \Bbbk g, g\otimes 1)\},\; \{(\Bbbk 1\oplus \Bbbk gx, 1\otimes 1)\}, \;\{(H_4, 1\otimes 1)\},
\]
\[
\{(\Bbbk 1 \oplus \Bbbk(g+\tau gx), 1\otimes 1), \,  (\Bbbk 1 \oplus \Bbbk(g - \tau gx), 1\otimes 1)\}_{\tau \in \Bbbk^\times}.
\]
When $\lambda \neq 0$, the classes are
\[
\{(\Bbbk, 1\otimes 1)\}, \;
\{(\Bbbk 1\oplus \Bbbk g, 1\otimes 1)\},  \;
\{(\Bbbk 1\oplus \Bbbk gx, 1\otimes 1)\}, \;
\{(H_4, 1\otimes 1)\},
\]
\[
\{(\Bbbk 1 \oplus \Bbbk(g+\tau gx), 1\otimes 1), \,  (\Bbbk 1 \oplus \Bbbk(g - \tau gx), 1\otimes 1)\}_{\tau \in \Bbbk^\times}.
\]
\end{theorem}

\begin{proof}
The quasitriangular left coideal subalgebras of $H_4$ when $\lambda=0$ are given in  Proposition~\ref{prop:H4-K}, and we will analyze this case in detail. The result for the $\lambda \neq 0$ case will follow by similar reasoning as the arguments below will not depend on the parameter $\lambda$. Note that by Corollary~\ref{cor:csa-dim-same}, quasitriangular coideal subalgebras that have different dimensions are not braided Morita equivalent. So, we only need to analyze the equivalence between those of dimension 1 and between those dimension~2.

Let us consider the representation of $\textnormal{Br}_2^{\textnormal{BC}}$ from  Proposition~\ref{prop:typeBC}:
\[
\begin{array}{rl}
 \rho_2^{(H_4)_{\textnormal{reg}}, B_{\textnormal{reg}}}(t)\colon H_4\otimes H_4 \otimes B &\longrightarrow H_4\otimes H_4 \otimes B\\[.2pc]
h\otimes h'\otimes b &\mapsto \sum_i h\otimes K_i h'\otimes K^i b,
\end{array}
\] where $K= \sum_i K_i \otimes K^i$ is the $K$-matrix of the left coideal subalgebra $B$. Its trace is $16\,\text{dim}_\Bbbk B$ when $K=1\otimes 1$, and its trace  is 0 when $K=g\otimes 1$.  
Now by Theorem \ref{thm:typeBC}, $(\Bbbk, 1\otimes 1)$ and  $(\Bbbk, g\otimes 1)$ are not braided Morita equivalent. Likewise, for any $\tau\in\Bbbk$, the pair $(\Bbbk 1 \oplus \Bbbk (g+\tau gx), 1\otimes 1)$  and $(\Bbbk 1 \oplus \Bbbk g, g\otimes 1)$, as well as the pair $(\Bbbk 1 \oplus \Bbbk g, g\otimes 1)$ and  $(\Bbbk 1 \oplus \Bbbk gx, 1\otimes 1)$, are not braided Morita equivalent. 

Moreover, for any $\tau\in\Bbbk$, the algebra $\Bbbk 1 \oplus \Bbbk (g+\tau gx)$ is semisimple \cite[Proposition~4.6]{CKS}, whereas the algebra $\Bbbk 1 \oplus \Bbbk gx$ is non-semisimple \cite[Proposition~4.3]{CKS}. So, these (quasitriangular) coideal subalgebras cannot be (braided) Morita equivalent (with respect to any $K$-matrix). 

Finally, we consider when $(\Bbbk 1 \oplus \Bbbk (g+\tau gx), 1\otimes 1)$ and $(\Bbbk 1 \oplus \Bbbk (g+\tau' gx), 1\otimes 1)$ are braided Morita equivalent for $\tau, \tau' \in \Bbbk$. Let $B_\tau: =\Bbbk 1 \oplus \Bbbk (g+\tau gx)$, with $y_\tau:=g+\tau gx$. Then, it is straightforward to see that $B_\tau \cong B_{\tau'}$ as $H_4$-comodule algebras if and only if $\tau' = \pm \tau$. Namely, an algebra morphism must send $y_\tau$ to $\pm y_{\tau'}$, and the $H_4$-comodule condition implies that $\tau' = \pm \tau$. On the other hand, \cite[Theorem~4.2]{GarciaMombelli} implies that $B_\tau$ and $B_{\tau'}$ are braided Morita equivalent if and only if there is a grouplike element $h \in G(H_4)$ such that $h B_\tau h^{-1} \cong B_{\tau'}$ as $H_4$-comodule algebras. Since $g B_\tau g^{-1} \cong B_{-\tau}$, we conclude that $(B_\tau, 1\otimes 1)$ and $(B_{\tau'}, 1\otimes 1)$ are braided Morita equivalent precisely when $\tau'=\pm \tau$. This completes the proof.
\end{proof}

%%%%%%%%%%%%%%%%%%%%%%%%%%
%%%%%%%%%%%%%%%%%%%%%%%%%%
%%%%%%%%%%%%%%%%%%%%%%%%%%

\section*{Acknowledgements} 

The authors are very grateful for the anonymous referee's time and careful consideration when reviewing our manuscript. Their feedback led us to add Lemma~\ref{lem:HsimpleCSA}, to adjust the proof of Proposition~\ref{prop:AM-S}(b), and to streamline the prose of Section~\ref{sec:results} and computations in Section~\ref{sec:examples}.

M.M. thanks the warm hospitality of C.W. during a visit to Rice University, where part of this work was done. M.M. was partially supported by Simons Foundation Award 889000 as part of the Simons
Collaboration on Global Categorical Symmetries. C.W. thanks M.M. for her aforementioned visit; it was fantastic. C.W. also thanks Jacob Kesten for discussions about Proposition~\ref{prop:typeD} and Question~\ref{ques:othertypes}. C.W. was partially supported by the US National Science Foundation grant \#DMS-2348833.

%%%%%%%%%%%%%%%%%%%%%%%%%%%%%%%%%%%%%%%%%%%%%%%%%%%%%%%%%%%%%%%%%%%%%%%%%%%%

\bibliography{MW-2023}
\bibliographystyle{amsrefs}%agsm or dcu

%%%%%%%%%%%%%%%%%%%%%%%%%%
%%%%%%%%%%%%%%%%%%%%%%%%%%
%%%%%%%%%%%%%%%%%%%%%%%%%%

\end{document}